%% file: main.tex
\documentclass[12pt]{amsart}
\usepackage{amsmath}
\usepackage{amssymb}
\usepackage{verbatim}
\usepackage{graphicx}
\usepackage{color}
\usepackage{dsfont}
\usepackage{hyperref}
\usepackage{environ}
\hypersetup{colorlinks=true, linkcolor=blue, citecolor=magenta}
\usepackage{url}
\usepackage{fancyhdr}
\usepackage{tikz-cd}
\usepackage{tikz}
\usetikzlibrary{shapes.geometric,arrows.meta}
\tikzcdset{arrow style=tikz}
\usepackage{calc}
\usepackage{stmaryrd}
\usepackage{quiver}

\NewEnviron{tikzineqn}[1][]{\vcenter{\hbox{\begin{tikzpicture}[#1]
            \BODY \end{tikzpicture}}}}

\usepackage[new]{old-arrows}  
\input{TestStyle.tikzstyles} 
\input{}
\tikzstyle{none}=[]
\pgfdeclarelayer{nodelayer}
\pgfdeclarelayer{edgelayer}
\pgfsetlayers{edgelayer,main,nodelayer}

\usepackage{array}
\usepackage{enumerate}
\usepackage{xfrac}
\usepackage{mathrsfs}   

\newcommand{\1}{\mathds{1}}
\newcommand{\bt}{\mathop{\boxtimes}\limits}
\newcommand{\dash}{\text{-}}
\newcommand{\id}{\mathrm{id}}
\newcommand{\Char}{\mathrm{char}}
\newcommand{\Id}{\mathrm{Id}}
\newcommand{\im}{\mathrm{im}}

\newcommand{\Hom}{\mathrm{Hom}}
\newcommand{\Fix}{\mathrm{Fix}}
\newcommand{\hofib}{\mathrm{hofib}}
\newcommand{\Inv}{\mathrm{Inv}}
\newcommand{\Irr}{\mathrm{Irr}}
\newcommand{\Mod}{\mathrm{Mod}}
\newcommand{\Bim}{\mathrm{Bim}}
\newcommand{\Mor}{\mathrm{Mor}}
\newcommand{\M}{\mathrm{Mor}_{2}}
\newcommand{\Rex}{\mathrm{Rex}}
\newcommand{\Fun}{\mathrm{Fun}}
\newcommand{\Triv}{\mathrm{Triv}}
\newcommand{\Glor}{\mathrm{Glor}}
\newcommand{\FPdim}{\mathrm{FPdim}}
\newcommand{\Gal}{\mathrm{Gal}}
\renewcommand{\Mor}{\mathrm{Mor}}
\newcommand{\Alg}{\mathrm{Alg}}
\newcommand{\Forg}{\mathrm{Forg}}

\newcommand{\aut}{\mathrm{Aut}}
\newcommand{\Aut}{\mathcal Aut}
\newcommand{\Out}{\mathcal Out}
\newcommand{\out}{\mathrm Out}
\newcommand{\Br}{\mathrm{Br}}
\newcommand{\brpic}{\mathrm{BrPic}}
\newcommand{\BrPic}{\mathcal Br\mathcal Pic}

\newcommand{\Tens}{\mathbf{Tens}}

\newcommand{\BrTens}{\mathbf{BrTens}}
\newcommand{\Pic}{\mathcal Pic}

\newcommand{\End}{\mathrm{End}}

\renewcommand{\Vec}{\mathrm{Vec}}
\newcommand{\Rep}{\mathrm{Rep}}

\renewcommand{\lim}{\mathop{\mathrm{lim}}}

\makeatletter
\g@addto@macro\th@plain{\thm@headpunct{}}
\g@addto@macro\th@definition{\thm@headpunct{}}
\g@addto@macro\th@remark{\thm@headpunct{}}
\makeatother

\makeatletter
\newtheorem*{rep@theorem}{\rep@title}
\newcommand{\newreptheorem}[2]{%
\newenvironment{rep#1}[1]{%
 \def\rep@title{#2 \ref{##1}}%
 \begin{rep@theorem}}%
 {\end{rep@theorem}}}
\makeatother

\newtheorem{theorem}{Theorem}[section]
\newreptheorem{theorem}{Theorem}
\newtheorem{proposition}[theorem]{Proposition}

\newtheorem{corollary}[theorem]{Corollary}
\newreptheorem{corollary}{Corollary}
\newtheorem{lemma}[theorem]{Lemma}

\theoremstyle{definition}
\newtheorem{definition}[theorem]{Definition}

\newtheorem{example}[theorem]{Example}

\newtheorem{remark}[theorem]{Remark}
\newtheorem{question}[theorem]{Question}

\newtheorem{conjecture}[theorem]{Conjecture}

\tikzset{Rightarrow/.style={double equal sign distance,>={Implies},->},
triple/.style={-,preaction={draw,Rightarrow}},
quadruple/.style={preaction={draw,Rightarrow,shorten >=0pt},shorten >=1pt,-,double,double
distance=0.2pt}}

\usepackage[foot]{amsaddr}
\usepackage{orcidlink}

\author[S. Sanford]{Sean Sanford\, \orcidlink{0000-0002-2439-3764}}
\address{School of Mathematics, The University of Edinburgh, Edinburgh, UK EH9 3FD}
\email{ssanford@ed.ac.uk}

\title{Putting the Brauer back in Brauer-Picard}
\date{}

\begin{document}

\begin{abstract}
    We establish a 6-term left exact sequence, involving Galois cohomology of the base field $\mathbb K$, and the Brauer-Picard groupoid of a fusion category.  This generalizes a result of Etingof, Nikshych, and Ostrik to the setting where $\mathbb K$ is not algebraically closed.
    Following their example, we use this exact sequence to compute examples of graded extensions of fusion categories over $\mathbb R$.
    Along the way, we establish several structural theorems regarding the duality morphisms for a fusion category as an object in the 4-category of braided tensor categories.
    The paper ends with a speculative look at a potential higher categorical explanation of the main result.
\end{abstract}

\maketitle
\tableofcontents

\section{Introduction}

In \cite{MR2677836}, Etingof, Nikshych, and Ostrik describe a map $\Phi$, from the group of invertible bimodules $\brpic(\mathcal C)$ for a fusion category $C$, to $\aut_{br}(\mathcal Z(\mathcal C))$, the group of braided autoequivalences of the Drinfeld center $\mathcal Z(\mathcal C)$.
Their work shows that this map $\Phi$ is an isomorphism, but their result is stronger than this.
These groups are just truncations of categorical groups: monoidal categories whose objects ($\pi_0$), and morphisms ($\pi_1$) have interacting group structures.
Their theorem \cite[Thm. 1.1]{MR2677836} actually establishes that this $\Phi$ comes from an equivalence of these categorical groups.
This equivalence is a valuable tool when trying to compute extensions of fusion categories by groups, which is a necessary component of a gauging procedure motivated by the connections between fusion categories and topological phases of matter.

... but what \emph{really} happened was that $\brpic(\mathcal C)$ was just the $\pi_0$ of a 2-categorical group $\BrPic(\mathcal C)$, and this had to first be truncated to $\pi_{\leq1}$ in order for $\Phi$ to be an equivalence.
The group $\pi_2\BrPic(\mathcal C)$ is canonically isomorphic to $\mathbb K^\times$, the group of units in the base field, so not much data was lost in removing this top level.
However, this issue of truncation hinted at the fact that this canonical map $\Phi$, which really is a quite general construction, should not be expected to be an equivalence in full generality.

In this paper, we will show that by simply changing the base field, $\pi_0\Phi$ can fail to be injective.
In order to investigate the resulting kernel, the most natural thing to do, from the homotopy theorist's perspective, is to take the homotopy fiber of $\Phi$.
We accomplish this, and manage to identify $\hofib(\Phi)$ in the following way.

\begin{reptheorem}{thm:fiber sequence}
    Suppose that $\mathcal C$ is a 2-separable (see Definition \ref{def:what it means to be separable fusion}) fusion category over $\mathbb K$ and $\Omega\mathcal Z\mathcal C:=\End_{\mathcal Z(\mathcal C)}(\1)=\mathbb K$.  There is a homotopy fiber sequence
    \[\BrPic(\Vec_{\mathbb K})\to\BrPic(\mathcal C)\xrightarrow{\Phi}\Aut_{br}\big(\mathcal Z(\mathcal C)\big)\,.\]
\end{reptheorem}
A similar result was announced in \cite[Thm. 6.3.7]{sanfordThesis}, and the theorem above amounts to a significant reduction in complexity of the hypotheses.

The 2-categorical group $\BrPic(\Vec_{\mathbb K})$ has $\pi_0\cong\Br(\mathbb K)\cong H^2(\mathbb K;\mathbb G_m)$.
This connection with the Brauer group was discussed in \cite[Prop. 4.9]{MR2677836} as a motivation for the name Brauer-Picard. Since Etingof, Nikshych, and Ostrik were working over an algebraically closed field $\mathbb K$ for the rest of the paper, this particular higher group was trivial at $\pi_0$, which meant that the kernel coming from the homotopy fiber was absent.

The main consequence of Theorem \ref{thm:fiber sequence} is the existence of a corresponding long exact sequence of homotopy groups.
With a little extra work, we are able to extend this sequence one additional term to the right by defining a map $\mathcal T$ (Definition \ref{def:the map T}), from $\aut_{br}(\mathcal Z(\mathcal C))$ to the group of invertible fusion categories up to Morita equivalence.
The group of such invertible Morita classes was identified in \cite[Thm. 5.9]{sanford2024invertiblefusioncategories} as being isomorphic to $H^3(\mathbb K;\mathbb G_m)$, which is often, but not always trivial (see Remark \ref{rem:good and bad fields}).
In particular, $H^3(\mathbb K;\mathbb G_m)=1$ whenever $\mathbb K=\overline{\mathbb K}$, so $\Vec_{\mathbb K}$ is the only invertible fusion category over such a field.
In this setting, the theorem below recovers the surjectivity of $\pi_0\Phi$ found in \cite[Thm. 1.1]{MR2677836}.

\begin{reptheorem}{thm:what's the cokernel?}
    Suppose that $\mathcal C$ is 2-separable fusion over $\mathbb K$, and $\Omega\mathcal Z(\mathcal C)=\mathbb K$.
    The long exact sequence of homotopy groups associated to the fibration of Theorem \ref{thm:fiber sequence} can be extended by appending $\mathcal T$ (shown below), though $\mathcal T$ itself is not necessarily surjective.
    \[
    \begin{tikzpicture}[x=3cm,y=1.2cm]
        \node (A) at (0,2) {$\mathbb K^\times$};
        \node (B) at (1,2) {$\mathbb K^\times$};
        \node (C) at (2,2) {$0$};
        \node (D) at (0,1) {$0$};
        \node (E) at (1,1) {$\mathrm{Inv}\big(\mathcal Z(\mathcal C)\big)$};
        \node (F) at (2,1) {$\mathrm{Aut}_{\otimes}(\mathrm{Id}_{\mathcal Z(\mathcal C)})$};
        \node (G) at (0,0) {$\mathrm{Br}(\mathbb K)$};
        \node (H) at (1,0) {$\mathrm{BrPic}(\mathcal C)$};
        \node (I) at (2,0) {$\mathrm{Aut}_{br}\big(\mathcal Z(\mathcal C)\big)$};
        \node (J) at (0,-1) {$H^3(\mathbb K;\mathbb G_m)$};
        \draw[->, rounded corners]
        (A) edge (B)
        (B) edge (C)
        (C) -- ++(0.6,0) -- ++(0,-.5) -- ++(-3.1,0) -- ++(0,-.5) -- ++(.1,0) edge (D)
        (D) edge (E)
        (E) edge  (F)
        (F) -- ++(0.6,0) -- ++(0,-.4) -- ++(-3.1,0) -- ++(0,-.6) -- ++(.1,0) edge (G)
        (G) edge (H)
        (H) -- node[above]{$\pi_0\Phi$} (I);
        \draw[->, rounded corners, orange] (I) -- ++(0.6,0) -- ++(0,-.5) -- node[below]{$\mathcal T$} ++(-3.1,0) -- ++(0,-.5) -- (J);
    \end{tikzpicture}
    \]
\end{reptheorem}

Note that $H^0(\mathbb K;\mathbb G_m)=\mathbb K^\times$, and $H^1(\mathbb K;\mathbb G_m)=0$ by Hilbert's Theorem 90.
This fiber sequence thus connects Galois cohomology and fusion categories over general fields.

Let us describe this map $\mathcal T$.
For a braided equivalence $F:\mathcal Z(\mathcal C)\to\mathcal Z(\mathcal C)$, we can construct a multifusion category
\[\mathcal T_F\;:=\;\mathcal C^{mp}\bt_{\mathcal Z(\mathcal C)}\mathcal Z(\mathcal C)_F\bt_{\mathcal Z(\mathcal C)}\mathcal C\,\]
which could be interpreted as `conjugation of $F$ by $\mathcal C$'.
Theorem \ref{thm:main invertibility theorem} (see below) guarantees that this category $\mathcal T_F$ is an invertible multifusion category, and thus represents some Morita class of invertible fusion categories.

This long (left) exact sequence is the primary tool for computing group extensions of fusion categories when the base field $\mathbb K$ is not algebraically closed.
In Example \ref{eg:Q+}, we describe a fusion category over $\mathbb R$, where this extra term $\Br(\mathbb R)\cong\mathbb Z/2\mathbb Z$ contributes to $\brpic(\mathcal C)$, and in Example \ref{eg:Q-} we describe a similar category where this term contributes to $\aut_{\otimes}(\Id_{\mathcal Z(\mathcal C)})$ instead.
These examples show that the 5-term exact sequence of Theorem \ref{cor:5-term LES} cannot be split into smaller sequences generically.

While building up to the proofs of these results, we establish various structural results that are interesting in their own right.
The careful reader may have already noticed the prevalence of the condition $\Omega\mathcal Z\mathcal C=\mathbb K$ in each of these theorems.
The significance of the unit in the Drinfeld center being split stems from the following computation.

\begin{reptheorem}{lem:main faithfulness lemma}
    For any 2-separable fusion category $\mathcal C$ over $\mathbb K$, the action of $\mathcal C\bt\mathcal C^{mp}$ on $\mathcal C$ is faithful if and only if $\Omega\mathcal Z(\mathcal C)=\mathbb K$.
\end{reptheorem}

This lemma relies on an interesting Galois theory result of Kuperberg \cite{kuperbergFiniteConnectedSemisimple2002} in the context of rigid bicategories.

In Section \ref{sec:The 4-cat Mor2}, we describe the 4-category $\M$ of Braided monoidal $\mathbb K$ linear categories that provides a unified language for our arguments.
Much work has gone into analyzing duality \cite{MR4228258,gwilliam2018dualsadjointshighermorita} and invertibility \cite{MR4302495} in this 4-category, and even defining $\M$ rigorously required a great deal of effort (see \cite{Scheimbauer2014FactorizationHA}, \cite{haugsengTheHigherMoritaCategory}, \cite{MR3590516}, \cite{dissertation}).
Using this 4-categorical language, we establish the following.

\begin{reptheorem}{thm:main invertibility theorem}
    For $\mathcal C$ a 2-separable fusion category over $\mathbb K$, $\mathcal C$ is invertible as a 1-morphism $\mathcal Z(\mathcal C)\to\Vec_{\mathbb K}$ in $\M$ if and only if $\Omega\mathcal Z(\mathcal C)=\mathbb K$.
\end{reptheorem}

This result is the main technical heart of the paper, and the proof Theorem \ref{thm:fiber sequence} relies on this invertibility in a critical way.
Since the composition of invertible morphisms is again invertible, the above result also explains why the map $\mathcal T$ produces invertible fusion categories.

Finally in Section \ref{sec:Conjectural extension}, we describe a conjecture of Corey Jones and David Reutter about a 2-categorical version of the fiber sequence established in \cite[Thm. 3.4]{MR3354332}.
Assuming this conjecture, we explain how our fiber sequence, including the extension by $\mathcal T$, can be recovered as a special case of theirs.

\vspace{4mm}

In addition to the author's algebraic interests, a secondary motivation for this work comes from the relation between fusion categories over the real numbers, and time reversal symmetry in topological quantum field theory.
Fusion categories are algebraic objects that describe systems of particles in a topological order.
When a fusion category is equipped with a braiding, it can describe (quasi-)particles in (2+1)D spacetime, and without the braiding it only describes particles in (1+1)D.
From the algebraic perspective, the fusion category is agnostic as to whether the objects are fundamental particles, or emergent quasi-particles, and thus they are valuable in condensed matter physics, where emergent quasi-particles often arise as localized excitations.

It is generally agreed that the operation of time reversal acts by some antiunitary operator.
If a system of (quasi-)particles is time reversal symmetric, then the antiunitarity should be represented in the fusion category as an autoequivalence that behaves as complex conjugation on scalars.
The collection of time reversal invariant (possibly composite) particles should then form a fusion category over the reals, which can be computed as the equivariantization with respect to the $\Gal(\mathbb C/\mathbb R)$ action, as described in \cite{etingofDescentAndForms}.
The results of this paper have implications for the analysis of time reversal symmetry performed in \cite{PhysRevB.93.235161}, and we intend to expand on these ideas in a following paper.

\subsection{Acknowledgments}
Thank you to Thibault Décoppet, Theo \linebreak Johnson-Freyd, and Corey Jones for carefully explaining their results to me.
Thanks to Julia Plavnik and to David Penneys for their advice, and to David Jordan for providing edits on an early draft, and for coming up with the catchy title.
The parts of this paper relating to invertibility in $\M$ were originally imagined as a follow-up paper to \cite{sanford2024invertiblefusioncategories}, and I want to thank Noah Snyder for giving me his blessing to write this up on my own.
A weaker version of Theorem \ref{thm:fiber sequence} first appeared in my PhD thesis \cite{sanfordThesis}, and I owe special thanks to Julia Plavnik for often reminding me that this paper should be written.

\section*{Funding sources}
    The research for this paper was partially funded by many different sources over many years, including National Science Foundation Grants No. DMS-2000093 and DMS-2154389, Simons Foundation award 888988 as part of the Simons Collaboration on Global Categorical Symmetries, and Engineering and Physical Sciences Research Council Grants No. 13462195\_13462197, and 9424181\_9424203.

\section{Preliminaries}

We will fix an arbitrary field $\mathbb K$, and this will be the base over which our 1- and 2-categories are constructed.

\begin{remark}
    We focus specifically on the finite dimensional setting.
    For ease of readability, we use notation that suppresses this assumption.
\end{remark}

The notation $\Vec_{\mathbb K}$ will refer to the category of finite dimensional vector spaces over $\mathbb K$.
More generally, for an algebra $A$ in $\Vec_{\mathbb K}$, $\Mod(A)$ will refer to the category of finite dimensional $A$ modules.

\begin{definition}\label{def:Separable algebra in a 1-cat}
    An algebra object $A$ in a monoidal category is said to be separable if its multiplication map $\mu:A\otimes A\to A$ admits a section $m^*:A\to A\otimes A$ as a map of $A$-$A$ bimodules.
\end{definition}

\begin{definition}
    A category $\mathcal C$ will be called 1-separable (over $\mathbb K$) if $\mathcal C\simeq\Mod(A)$, for some separable algebra $A$ in $\Vec_{\mathbb K}$.
\end{definition}

The tensor product of separable algebras is always a separable algebra, but the tensor product of semisimple algebras can fail to be semisimple, as the following example shows.
\begin{example}\label{eg:InseparableSemisimple}
    If $\mathbb L$ is an inseparable field extension over $\mathbb K$, then $\mathbb L$ is a semisimple algebra in $\Vec_{\mathbb K}$, but $\mathbb L\otimes_{\mathbb K}\mathbb L$ is nonsemisimple.
\end{example}
In order to avoid these problems, authors often work over perfect fields, which makes 1-separability equivalent to semisimplicity.

The paper \cite{gaiotto2019condensationshighercategories} proposes an inductive definition of separable algebra in all categorical dimensions, as well has a higher analogue of idempotent- (also called Karoubi-) completion.
A thorough treatment of idempotent completion for linear 2-categories can be found in \cite[Sec. 1.3]{douglasFUSION2CATEGORIESSTATESUM}.
These higher notions of separability once again behave well under tensor products.

With these benefits in mind, and in order to effectively argue about higher linear categories, we choose to abandon semisimplicity in favor of separability.
This leads to the following definition.

\begin{definition}[{cf. \cite[Ex. 1.5.2]{decoppetRigidAndSeparable}}\footnote{Décoppet uses the word `perfect' for what we are calling 1-separable.}]
    The notation $2\Vec_{\mathbb K}$ will refer to the 2-category of categories that are 1-separable over $\mathbb K$, together with linear functors and natural transformations.
\end{definition}

There is a larger 2-category $\Rex_{\mathbb K}$, consisting of all finitely cocomplete $\mathbb K$-linear categories, and colimit preserving (aka right-exact, hence Rex) functors.
There is an embedding $2\Vec_{\mathbb K}\hookrightarrow\Rex_{\mathbb K}$ that is an equivalence on hom categories.
These 2-categories admit a symmetric monoidal structure known as the Deligne-Kelly tensor product $\bt:=\boxtimes_{\mathbb K}$.
There are many different settings in which this product makes sense (see e.g. \cite{lopezFrancoTensorProducts}), and we have chosen our 2-categories so that they are well-behaved with respect to this operation.

\begin{definition}[universal version]\label{def:UniversalDeligne}
    The Deligne-Kelly tensor product $\mathcal C\bt\mathcal D$ of two finitely cocomplete $\mathbb K$-linear categories $\mathcal C$ and $\mathcal D$ is a finitely cocomplete $\mathbb K$-linear category together with a bilinear functor $\mathcal C\times\mathcal D\to\mathcal C\bt\mathcal D$ that induces an equivalence
    \[\Rex_{\mathbb K}(\mathcal C\bt\mathcal D,\mathcal E)\simeq\Rex^{bil}_{\mathbb K}(\mathcal C,\mathcal D\,;\mathcal E)\,,\]
    where $\Rex^{bil}_{\mathbb K}(\mathcal C,\mathcal D\,;\mathcal E)$ is the category of $\mathbb K$-bilinear functors.
\end{definition}

\begin{example}[Eilenberg-Watts]\label{eg:Eilenberg-Watts}
    If $\mathcal C\simeq\Mod(A)$, and $\mathcal D\simeq\Mod(B)$ for two finite dimensional algebras over $\mathbb K$, then $\mathcal C\bt\mathcal D\simeq\Mod(A\otimes B)$.
    In particular, the Deligne tensor product of two 1-separable categories is again 1-separable.
\end{example}

When representing algebras are given, the above example serves as a convenient computational tool.
The following construction is better suited to computing the Deligne-Kelly tensor product when we know the objects and morphisms.

\begin{definition}[constructive version]\label{def:ConstructiveDeligne}
    The naïve tensor product $\mathcal C\otimes\mathcal D$ of two $\mathbb K$-linear categories $\mathcal C$ and $\mathcal D$ is the category whose objects are pairs $(C,D)$, with $C\in\mathcal C$ and $D\in\mathcal D$, and whose morphisms are determined by the rule
    \[\mathcal C\otimes\mathcal D\big((C,D)\,,\,(C',D')\big):=\mathcal C(C,C')\otimes_{\mathbb K}\mathcal D(D,D')\,.\]
    The Deligne-Kelly product $\mathcal C\bt\mathcal D$ is the completion of $\mathcal C\otimes\mathcal D$ under finite colimits.
    For a given pair $(C,D)$ its image in $\mathcal C\bt\mathcal D$ under the canonical functor $\mathcal C\times\mathcal D\to\mathcal C\bt\mathcal D$ will be called a simple tensor and be denoted by $C\bt D$.
    If $p\in\End(C\bt D)$ is a projection, then we will write $C\boxtimes^pD$ for the image $\im(p)$.
\end{definition}

\begin{example}
    For $\mathbb H$ the quaternion algebra over $\mathbb R$, the projection $p=\tfrac14(1\otimes1-i\otimes i-j\otimes j-k\otimes k)$ determines a simple object $\mathbb H\boxtimes^p\mathbb H$ inside
    \[\Mod(\mathbb H)\boxtimes_{\mathbb R}\Mod(\mathbb H)\simeq\Mod(\mathbb H\otimes_{\mathbb R}\mathbb H)\simeq\Mod\big(M_4(\mathbb R)\big)\simeq\Vec_{\mathbb R}\,,\]
    and all the other simple projections give simple objects that are isomorphic to this one.
\end{example}

We now turn to the topic of monoidal structures on $\mathbb K$-linear categories.
The collection of all linear monoidal categories is very large, and so, in order to prove stronger results, we will restrict our attention to fusion categories.
Whereas fusion categories are customarily defined over algebraically closed fields, the main objects in this paper are fusion categories over non-algebraically closed fields. Hence, the definition of fusion must be adapted, and we follow the conventions of \cite{MR4806973}.

\begin{definition}
    A multifusion category over $\mathbb K$ is a 1-separable (over $\mathbb K$) rigid monoidal category, where the tensor product of morphisms is required to be bilinear over $\mathbb K$. A multifusion category is called fusion if the monoidal unit is a simple object.
\end{definition}

\begin{example}\label{eg:Bim(C)}
    The complex numbers $\mathbb C$ are a separable algebra in $\Vec_{\mathbb R}$, and the category $\Bim_{\Vec_{\mathbb R}}(\mathbb C)$ of bimodules for this algebra admits a monoidal product in the form of $\otimes_{\mathbb C}$.
    This monoidal structure makes the category fusion over $\mathbb R$.
    
    However, this category is not fusion over $\mathbb C$, because the tensor product is not bilinear over $\mathbb C$.
    To see this, consider the simple bimodule $\mathbb C_\sigma$, whose underlying object is $\mathbb C$, and where the left and right actions of $\mathbb C$ differ by complex conjugation.
    Left multiplication by complex scalars defines an isomorphism $\End(\mathbb C_\sigma)\cong\mathbb C$, but then it follows that $(z\cdot\id_{\mathbb C_\sigma})\otimes_{\mathbb C}(w\cdot\id_{\mathbb C_\sigma})=z\overline{w}\cdot(\id_{\mathbb C_{\sigma}}\otimes_{\mathbb C}\id_{\mathbb C_{\sigma}})$.
\end{example}

The requirement that the tensor product be bilinear over $\mathbb K$ is equivalent to the statement that $\otimes$ descends to a well-defined functor $\mathcal C\bt\mathcal C\to\mathcal C$, and this is what allows such categories to be interpreted as algebra objects in $2\Vec_{\mathbb K}$.
We refer the reader to \cite{decoppetRigidAndSeparable} for the definition and basics of algebra objects in linear 2-categories.

\begin{definition}[see {\cite[Def.s 2.1.1 and 2.1.7]{decoppetRigidAndSeparable}}]\label{def:Rigid&Separable}
    An algebra object $A$ in $2\Vec_{\mathbb K}$ is said to be rigid if its multiplication $\otimes:A\bt A\to A$ admits a right adjoint $I:A\to A\bt A$ as $A$-$A$ bimodule 1-morphisms.
    If the counit for the adjunction $\otimes\dashv I$ admits a section as $A$-$A$ bimodule 2-morphisms, then $A$ is said to be separable.
\end{definition}

This separability for an algebra object in a 2-category is a categorification of Definition \ref{def:Separable algebra in a 1-cat}.
In a 2-category, there are two layers of adjoints to consider: having the first adjoint means it is rigid, and having both means it is separable.

It is well-known (see, e.g. \cite[Def-Prop 1.3]{MR4228258}) that an algebra $A$ in $2\Vec_{\mathbb K}$ is rigid in the above sense if and only if $A$ is a rigid monoidal category.
In this way, the following proposition could alternatively be used as the definition of multifusion categories.

\begin{proposition}
    A multifusion category over $\mathbb K$ is the same thing as a rigid algebra object in $2\Vec_{\mathbb K}$.
\end{proposition}

The next issue we encounter when working over arbitrary fields is that the Drinfeld center of a fusion category is not necessarily fusion.

\begin{definition}\label{def:DrinfeldCenter}
    The Drinfeld center $\mathcal Z(\mathcal C)$ is $\Fun_{\mathcal C\text{-}\mathcal C}(\mathcal C,\mathcal C)$, the category of $\mathcal C$-$\mathcal C$ bimodule endofunctors of $\mathcal C$ and bimodule natural transformations.
\end{definition}

The objects in $\mathcal Z(\mathcal C)$ are functors that are equipped with bimodule structure maps (which we do not recall in detail here, see \cite[Def. 7.2.1 and Ex. 7.4.3]{MR3242743}).
The underlying functor $F$ is completely determined by simply evaluating at one: $F\mapsto F(\1)$, and this defines a monoidal forgetful functor $\Forg:\mathcal Z(\mathcal C)\to\mathcal C$.
The bimodule structure maps can be used to construct the following isomorphism
\[\beta^{F(\1)}_X:\;F(\1)\otimes X\cong F(\1\otimes X)\cong F(X)\cong F(X\otimes\1)\cong X\otimes F(\1)\,,\]
that is referred to as the half-braiding (over $X$) associated to the object $F(\1)$.
This collection is natural in $X$, and satisfies a further coherence axiom that follows from the bimodule coherence axioms.

If we set $Z=F(\1)$, then the pair $(Z,\beta^Z)$ defines a bimodule functor $Z\otimes(-)$, that is (bimodule) naturally isomorphic to $F$ itself.
In this way, we can think of objects in $\mathcal Z(\mathcal C)$ either as bimodule functors, or as objects equipped with half-braidings, and we will make liberal use of this equivalence.

\begin{example}\label{eg:NonSemisimpleCenters}
    Let $G$ be a finite group, and define $\Vec_{\mathbb K}(G)$ to be the category of finite dimensional $G$-graded vector spaces.
    When $\Char(\mathbb K)\nmid|G|$, the simple objects in $\mathcal Z(\Vec_{\mathbb K}(G))$ are classified by pairs $(C,V)$, where $C$ is a conjugacy class in $G$, and $V$ is an irreducible representation of the stabilizer of any element of $C$.
    When $\Char(\mathbb K)\mid|G|$, the description is more complicated.

    Regardless of the characteristic, it is always true that $\mathcal Z(\Vec_{\mathbb K}(G))$ contains $\Rep_{\mathbb K}(G)$ as a full subcategory.
    It follows that when $\Char(\mathbb K)\mid|G|$, $\mathcal Z(\mathcal C)$ is non-semisimple, despite the fact that $\Vec_{\mathbb K}(G)$ is 1-separable.
\end{example}

This issue is fixed by the following result:

\begin{proposition}[{cf. \cite[Cor. 3.1.7]{decoppetRigidAndSeparable}}]\label{prop:Separable<=>Z is LocSep}
    A multifusion category $\mathcal C$ is separable as a rigid algebra object in $2\Vec_{\mathbb K}$ (see Definition \ref{def:Rigid&Separable}), if and only if $\mathcal Z(\mathcal C)$ is 1-separable.
\end{proposition}

\begin{proof}
    The argument in \cite[Prop. 3.1.2]{decoppetRigidAndSeparable} when applied to $\mathfrak C=2\Vec_{\mathbb K}$ is stronger, because each of the categories $\Hom_{\mathfrak C}(C,N)$ are 1-separable.
    This shows that all the hom categories in $\Mod_{2\Vec_{\mathbb K}}(A)$ are 1-separable.
    The rest of the arguments \emph{loc. cit.} compile to show that, for a multifusion category $\mathcal C$, all the hom categories in $\Mod_{2\Vec_{\mathbb K}}(\mathcal C\bt\mathcal C^{mp})$ are 1-separable.
    In particular,
    \[\mathcal Z(\mathcal C)\;=\;Fun_{\mathcal C\dash\mathcal C}(\mathcal C,\mathcal C)\simeq\End_{\Mod_{2\Vec_{\mathbb K}}(\mathcal C\bt\mathcal C^{mp})}(\mathcal C)\]
    is one such hom category.
    
    The converse is immediate, seeing as it is a special case of the original version of \cite[Cor. 3.1.7]{decoppetRigidAndSeparable}.
\end{proof}

Note that $\mathcal Z(\mathcal C)$ being 1-separable is the same as saying that $\mathcal Z(\mathcal C)$ is an object in $2\Vec_{\mathbb K}$.
This is a property that we need for our later computations, so we pause a moment to clarify the terminology.

\begin{definition}\label{def:what it means to be separable fusion}
    A multifusion category $\mathcal C$ is 2-separable if it is separable as a rigid algebra in $2\Vec_{\mathbb K}$.
\end{definition}

\begin{remark}
    When considering a monoidal $\mathbb K$-linear category, 1-separability is a property of the underlying category, whereas 2-separability depends on the monoidal structure.
    Note that according to our definitions, multifusion categories are always 1-separable.
    Proposition \ref{prop:Separable<=>Z is LocSep} can be rephrased as saying: 
    \[\Bigg(\;\mathcal C \text{ is 2-separable}\;\Bigg)\;\;\iff\;\;\Bigg(\;\mathcal Z(\mathcal C) \text{ is 1-separable}\;\Bigg)\;.\]
    We hope that the clarity provided by this notational choice makes up for the aesthetic loss resulting from the proliferation of 1-s and 2-s.
\end{remark}

If $A$ is an algebra object internal $\mathcal C$, then we will write $\mathcal C_A$ for the category of right $A$ modules internal to $\mathcal C$, ${}_A\mathcal C$ for left modules, and ${}_A\mathcal C_A$ for bimodules.
By \cite[Prop. A.23]{MR4806973}, if $A$ is separable, then ${}_A\mathcal C$, $\mathcal C_A$, and ${}_A\mathcal C_A$ are 1-separable, and if ${}_A\mathcal C_A$ is 1-separable, then $A$ is separable.
More can be said when the ambient category $\mathcal C$ is also 2-separable.

\begin{theorem}\label{thm:LocSep&Sep=>Sep}
    Let $A$ be an algebra in a multifusion category $\mathcal C$.  If ${}_A\mathcal C$ (or $\mathcal C_A$) is 1-separable, and $\mathcal C$ is 2-separable, then $A$ is separable.
\end{theorem}

\begin{proof}
    The right adjoint $I:\mathcal C\to\mathcal C\bt\mathcal C$ to the tensor product determines a similar splitting
    \[{}_A\mathcal C_A\simeq{}_A\mathcal C\bt_{\mathcal C}\mathcal C_A\to{}_A\mathcal C\bt\mathcal C_A\,,\]
    which we will also denote by $I$, since it's the same functor on the underlying objects.
    The category $\mathcal C_A$ is canonically equivalent to $({}_A\mathcal C)^{op}$, and therefore both are 1-separable.
    It follows from Example \ref{eg:Eilenberg-Watts} that ${}_A\mathcal C\bt\mathcal C_A$ is also 1-separable.

    The object $I(A)$ in ${}_A\mathcal C\bt\mathcal C_A$ must admit a decomposition into simple objects, so by Definition \ref{def:ConstructiveDeligne} we can find an isomorphism
    \[I(A)\cong\bigoplus_iM_i\bt^{p_i}N_i\,,\]
    because 1-separability implies all simple objects will be summands of simple tensors.
    We can now find inclusions that split the projections $p_i$, and this allows us to view $I(A)$ as a summand of a sum of simple tensors.
    Further, semisimplicity of ${}_A\mathcal C$ implies that each $M_i$ is a summand of $A\otimes M_i$, and similarly for the $N_i$.
    This allows us to compose to find a map
    \begin{gather*}
        I(A)\hookrightarrow\bigoplus_iM_i\bt N_i\hookrightarrow\bigoplus_i(A\otimes M_i)\bt (N_i\otimes A)\cdots\\
        \cdots\twoheadrightarrow A\otimes\left(\bigoplus_iM_i\bt^{p_i}N_i\right)\otimes A\cong A\otimes I(A)\otimes A\,,
    \end{gather*}
    which is a map of $A$-$A$ bimodules.

    Let $\xi:\Id_{\mathcal C}\to\otimes\circ I$ be the $\mathcal C$ bimodule splitting of the counit $\epsilon:\otimes\circ I\to\Id_{\mathcal C}$.
    It follows that $\xi_A:A\to(\otimes\circ I)(A)$ is a map of $A$-$A$ bimodules.
    Using our previous composition, we can write down a map
    \begin{gather*}
        A\xrightarrow{\xi_A}(\otimes\circ I)(A)\to\otimes\big(A\otimes I(A)\otimes A\big)\cong A\otimes (\otimes\circ I)(A)\otimes A\cdots\\
        \cdots\xrightarrow{\id\otimes\epsilon_A\otimes \id}A\otimes A\otimes A\xrightarrow{\mu_A\otimes\id_A}A\otimes A\,.
    \end{gather*}
    This is our desired splitting $\Delta_A:A\to A\otimes A$.
    Our choice of the maps $M_i\to A\otimes M_i$ as splittings of the module actions, together with the fact that $\epsilon_A$ is a map of $A$-$A$ bimodules implies that $\mu_A\circ\Delta_A=\epsilon_A\circ\xi_A=\id_A$.
\end{proof}

If a monoidal $\mathbb K$-linear category has finite colimits, then we can form relative products of internal modules over internal algebras by taking the coequalizers as in the diagram below.
\[
    \begin{tikzcd}[column sep=15mm]
        M\otimes A\otimes N \ar[r, "\mu_M\otimes\id_N",shift left]\ar[r, "\id_M\otimes\mu_N"',shift right] & M\otimes N \ar[r, dotted] & M\otimes_AN\,.
    \end{tikzcd}
\]

There is also a relative version of the Deligne-Kelly tensor product, which has a universal property that categorifies the above coequalizer.

\begin{definition}\label{def:C-balFunctors&RelativeDeligne}
    Given an algebra object $\mathcal C$ and left and right $\mathcal C$ modules $\mathcal M$ and $\mathcal N$ in $\Rex_{\mathbb K}$, a $\mathcal C$-balanced functor $(F,\text{bal}^F):\mathcal M\times\mathcal N\to\mathcal E$ is a bilinear functor $F$, cocontinuous in each variable, together with an isomorphism
    \[
        \text{bal}^F_{M,C,N}:F\big(M\triangleleft C\;,\; N\big)\to F\big(M\;,\;C\triangleright N\big)\,,
    \]
    called a $\mathcal C$-balancing that is natural in $M$, $C$, and $N$, and subject to the coherence conditions
    \[
    \begin{tikzcd}[ampersand replacement=\&]
    	{F\big((M\triangleleft C)\triangleleft C'\;,\; N\big)} \&\& {F\big((M\triangleleft C)\;,\;(C'\triangleright N)\big)} \\
    	{F\big(M\triangleleft (C\otimes C')\;,\; N\big)} \\
    	{F\big(M\;,\; (C\otimes C')\triangleright N\big)} \&\& {F\big(M\;,\; C\triangleright(C'\triangleright N)\big),}
    	\arrow["{\text{bal}^F_{M\triangleleft C,C',N}}", from=1-1, to=1-3]
    	\arrow["{\text{bal}^F_{M,C,C'\triangleright N}}", from=1-3, to=3-3]
    	\arrow[from=2-1, to=1-1]
    	\arrow["{\text{bal}^F_{M,C\otimes C',N}}"', from=2-1, to=3-1]
    	\arrow[from=3-1, to=3-3]
    \end{tikzcd}
    \]
    \[
        \begin{tikzcd}[ampersand replacement=\&]
        	{F(M\triangleleft\1\,,\,N)} \&\& {F(M\,,\,\1\triangleright N)} \\
        	\& {F(M\,,\,N)}
        	\arrow["{\text{bal}^F_{M,\1,N}}", from=1-1, to=1-3]
        	\arrow[from=1-1, to=2-2]
        	\arrow[from=1-3, to=2-2]
        \end{tikzcd}
    \]
    which enforce compatibility with the relevant module structure maps.

    A $\mathcal C$-balanced natural transformation is a natural transformation of the underlying functors that intertwines that $\mathcal C$-balancings.
    Together, these form a category $\Rex_{\mathbb K}^{\mathcal C\text{-bal}}(\mathcal M,\mathcal N;\mathcal E)$.
    
    The relative Deligne-Kelly tensor product $\mathcal M\boxtimes_{\mathcal C}\mathcal N$ is a category in $\Rex_{\mathbb K}$ together with a bilinear, finitely cocontinuous (in each variable) functor $Q:\mathcal M\times\mathcal N\to\mathcal M\boxtimes_{\mathcal C}\mathcal N$, restriction along which induces an equivalence
    \[Q^*:\Rex_{\mathbb K}(\mathcal M\boxtimes_{\mathcal C}\mathcal N,\mathcal E)\xrightarrow{\simeq}\Rex^{\mathcal C\text{-bal}}_{\mathbb K}(\mathcal M,\mathcal N\,;\mathcal E)\,.\]
\end{definition}

Note the similarity between Definition \ref{def:C-balFunctors&RelativeDeligne} and Definition \ref{def:UniversalDeligne}.
In fact, the triangle relation for $\mathcal C$-balancings implies that $\mathcal M\boxtimes_{\Vec_{\mathbb K}}\mathcal N$ is canonically equivalent as a category to $\mathcal M\boxtimes\mathcal N$.

\section{The 4-category \texorpdfstring{$\M$}{Mor\_2}}\label{sec:The 4-cat Mor2}

Here we discuss a 4-category built out of braided monoidal categories over $\mathbb K$, and their various bimodules.
A version of this category was described in \cite{MR4228258} and \cite{MR4302495}.
The version that we are concerned with here is more restrictive, so to be explicit, we will give both a concise technical definition, as well as a more verbose expanded description.

In brief, the $E_n$-Morita category $\Mor_n(\mathcal S)$ built from an $(\infty,d)$-category $\mathcal S$ consists of $E_n$-algebra objects in $\mathcal S$, bimodules valued in $E_{n-1}$-algebras, bimodules of bimodules valued in $E_{n-2}$-algebras, ..., bimodules of bimodules of... bimodules valued in $\mathcal S$, 1-morphisms of such iterated bimodules, 2-morphisms of such, ... , and $d$-morphisms of such.
In other words, it is an $(\infty,d+n)$-category, the first $n$ levels of which consist of algebras and iterated bimodules, and with the last $d$ levels consisting of bimodule structure preserving $(k-n)$-morphisms for $n+1\leq k\leq n+d$.

The machinery that allowed for a rigorous treatment of these higher Morita categories was first developed\footnote{Many of the references in this paragraph refer to the construction as $\Alg_n$.} in \cite{Scheimbauer2014FactorizationHA}.
Further refinements followed in \cite{haugsengTheHigherMoritaCategory}, \cite{MR3590516}, and \cite{gwilliam2018dualsadjointshighermorita}.
A key issue was that early iterations of this construction had top-level bimodules come equipped with a pointing, which was somewhat unnatural from an algebraist's perspective (see e.g. \cite[Sec. 1.4]{gwilliam2018dualsadjointshighermorita}).
This issue was finally resolved in \cite{dissertation}.
Specifically we are using the `even higher pointless' version described in Section 9.2 loc. cit.

\begin{definition}
    The 4-category $\M$ is the pointless $E_2$-Morita category $\Mor_2^{pl}(\Rex_{\mathbb K})$.
\end{definition}

While the research described above allows us to talk about higher Morita categories in a wide array of contexts, the most serious application of this theory has been in the context of $\M$ and $\Mor_1$.
In terms of our ability to perform technical computations in this specific instance, the majority of the practical results that we will employ were established\footnote{These references use the name $\BrTens_{\mathbb K}:=\Mor_2^{pl}(\mathcal Pr_{\mathbb K})$, and $\Tens_{\mathbb K}:=\Mor_1^{pl}(\mathcal Pr_{\mathbb K})$.  Here $\mathcal Pr_{\mathbb K}$ is the 2-category of locally presentable $\mathbb K$-linear categories, but we are working in the smaller setting of $\mathcal S=\Rex_{\mathbb K}$.} in \cite{douglasDualizable}, \cite{MR4228258}, and \cite{MR4302495}.
We now follow their example and provide a nuts-and-bolts description of $\M$.

\begin{definition}[unpacked, cf. {\cite[Def. 2.10]{MR4302495}}]
    The 4-category $\M$ has the following morphisms (all relevant categories and morphisms are in $\Rex_{\mathbb K}$):
    \begin{itemize}
        \item Objects: Braided monoidal categories.
        \item 1-Morphisms $\mathcal A\to\mathcal B$: Monoidal categories $\mathcal C$ equipped with a braided functor $F_{\mathcal C}:\mathcal A\bt\mathcal B^{rev}\to\mathcal Z(\mathcal C)$
        \item 2-Morphisms $\mathcal C\to\mathcal D$: $\mathcal C$-$\mathcal D$ bimodules $\mathcal M$, equipped with an $\mathcal A\bt\mathcal B^{rev}$-central structure, i.e. a natural isomorphism
        \[F_{\mathcal C}(X)\triangleright_{\mathcal C}M\to M\triangleleft_{\mathcal D}F_{\mathcal D}(X)\,,\]
        satisfying the relation
        \[
            \begin{tikzcd}[ampersand replacement=\&]
        	{F_{\mathcal C}(X\otimes Y)\triangleright M} \& {M\triangleleft F_{\mathcal D}(X\otimes Y)} \\
        	{F_{\mathcal C}(X)\triangleright \big(F_{\mathcal C}(Y)\triangleright M\big)} \& {\big(M\triangleleft F_{\mathcal D}(X)\big)\triangleleft F_{\mathcal D}(Y)} \\
        	{F_{\mathcal C}(X)\triangleright \big(M\triangleleft F_{\mathcal D}(Y)\big)} \& {\big(F_{\mathcal C}(X)\triangleright M\big)\triangleleft F_{\mathcal D}(Y)\,.}
        	\arrow[from=1-1, to=1-2]
        	\arrow[from=1-1, to=2-1]
        	\arrow[from=2-1, to=3-1]
        	\arrow[from=2-2, to=1-2]
        	\arrow[from=3-1, to=3-2]
        	\arrow[from=3-2, to=2-2]
            \end{tikzcd}
        \]
        \item 3-Morphisms $\mathcal M\to\mathcal N$: Bimodule functors $F:\mathcal M\to\mathcal N$ whose structure maps are coherent with respect to the $\mathcal A\bt\mathcal B^{rev}$-central structure.
        \item 4-Morphisms $F\to G$: Bimodule natural transformations $\mu:F\to G$.
    \end{itemize}
    Composition of 1- and 2-morphisms comes from relative Deligne tensor product (see Definition \ref{def:C-balFunctors&RelativeDeligne}), which exists by \cite[Rmk. 3.15]{MR3847209}, and composition of 3- and 4-morphisms is ordinary composition of bimodule functors and natural transformations.
    The central structures on a composition of central functors is constructed in a manner analogous to that for bimodule structures.
\end{definition}

Although we would prefer to build $\M$ out of $2\Vec_{\mathbb K}$ instead of the larger $\Rex_{\mathbb K}$, this is not possible, because $2\Vec_{\mathbb K}$ is not closed under relative tensor product.
Indeed, $\Vec_{\mathbb K}\boxtimes_{\Vec_{\mathbb K}(G)}\Vec_{\mathbb K}\simeq\Rep_{\mathbb K}(G)$, and so we encounter the same problem outlined in Example \ref{eg:NonSemisimpleCenters}.
The results in \cite{decoppet2024classificationfusion2categories} describe a `fully separable' sub-4-category built out of the 2-separable fusion categories and 1-separable modules, but this is not necessary for our computations.

For our purposes here, we will content ourselves with making arguments about 2-separable (braided) fusion categories, thinking of them as the well-behaved objects inside the greater chaos that is the gigantic 4-category $\M$.
For this purpose, we close the section with an important fact.

\begin{proposition}
    The relative Deligne product $\mathcal M\boxtimes_{\mathcal C}\mathcal N$ of two 1-separable module categories over a 2-separable multifusion category is again 1-separable.
\end{proposition}

\begin{proof}
    Using Ostrik's theorem \cite{ostrikModuleCatsWeakHopf}, we can find an algebra object $A$ in $\mathcal C$ for which $\mathcal M\simeq {}_A\mathcal C$ as right $\mathcal C$ module categories.
    By Theorem \ref{thm:LocSep&Sep=>Sep}, $A$ is a separable algebra in $\mathcal C$.
    Similarly for $\mathcal N$, we can find some separable algebra $B$ for which $\mathcal N\simeq\mathcal C_B$ as left $\mathcal C$ module categories.
    
    By \cite[Thm. 3.3(2)]{douglasBalancedTensorProduct2019} the category $\mathcal M\boxtimes_{\mathcal C}\mathcal N$ can be realized as ${}_A\mathcal C_B$, the category of $A$-$B$ bimodules internal to $\mathcal C$.
    A straightforward modification of the argument in the proof of \cite[Thm. A.23]{MR4806973} shows that separability of $A$ and $B$ implies that ${}_A\mathcal C_B$ must also be 1-separable.
\end{proof}

\section{Invertibility conditions}

Here we investigate the invertibility of certain 1-morphisms and 2-morphisms in $\M$.
The facts collected here will be the main tools for proving Theorem \ref{thm:main invertibility theorem}.
Throughout this and following sections, we will use the notation $\mathcal Z:=\mathcal Z(\mathcal C)$ to simplify statements and formulas.

The following result from the theory of 2-categories is well-known.
\begin{proposition}[{cf. \cite[Lemma 2.24]{MR4302495}}]\label{prop:inductive invertibility}
    A 1-morphism $\mathcal C:\mathcal B\to\mathcal A$ in a 2-category is invertible if, and only if, it is adjointable, and the unit and counit maps for this adjunction are isomorphisms.
\end{proposition}

Let's specialize to the case where $\mathcal C$ is a 2-separable fusion category over $\mathbb K$ so that, in particular, $\mathcal Z:=\mathcal Z(\mathcal C)$ is fusion as well.
Let us think of $\mathcal C$ as a $1$-morphism $\mathcal Z\to\Vec_{\mathbb K}$ in $\M$.
By the results in \cite[Sec. 3.11]{douglasDualizable}, the dual to $\mathcal C$ in $\M$ does exist, and it is given by the monoidal opposite, which we will denote as $\mathcal C^{mp}$.
The unit $\eta$ (or coevaluation) and counit $\epsilon$ (or evaluation) $2$-morphisms for the adjunction $\mathcal C\dashv\mathcal C^{mp}$ are
\begin{align}
    \eta:&=\mathcal C\,,\hspace{1mm}\text{as a }(\Vec_{\mathbb K},\Vec_{\mathbb K})\text{-central }(\mathcal C^{mp}\bt_{\mathcal Z}\mathcal C,\Vec_{\mathbb K})\text{ bimodule, \&}\\
    \epsilon:&=\mathcal C\,,\hspace{1mm}\text{as a }(\mathcal Z,\mathcal Z)\text{-central }(\mathcal Z,\mathcal C\bt\mathcal C^{mp})\text{ bimodule}\,.
\end{align}

By combining Proposition \ref{prop:inductive invertibility} with an inductive truncation argument, the invertibility of $\mathcal C$ as a $1$-morphism in $\M$ can be verified by proving that $\eta$ and $\epsilon$ are invertible bimodule categories.

It turns out that module categories for multifusion categories can often be upgraded to invertible bimodules.
We will need some notation to make sense of this procedure.

\begin{definition}[{\cite[Def. 7.12.2]{MR3242743}}]
    Given a monoidal category $\mathcal D$, and a left or right module category $\mathcal M$, the module dual\footnote{This construction is also sometimes called the centralizer of $\mathcal D$ in $\End(\mathcal M)$.} (of $\mathcal D$, with respect to $\mathcal M$) is
    \[\mathcal D^*_{\mathcal M}:=\Fun_{\mathcal D}(\mathcal M,\mathcal M)\;.\]
\end{definition}

The module dual has a monoidal structure coming from composition of functors, and naturally acts on $\mathcal M$ from the left.

\begin{theorem}[{cf. \cite[Ex. 7.12.8 \& Theorem 7.12.11]{MR3242743}}]\label{thm:faithful to invertible}
    Suppose $\mathcal D$ is a 2-separable multifusion category, and $\mathcal M$ is a 1-separable and faithful right $\mathcal D$ module category.
    By endowing it with the corresponding $(\mathcal D^*_{\mathcal M},\mathcal D)$ bimodule structure, $\mathcal M$ becomes an invertible bimodule category.
    Dually, if $\mathcal M$ is a 1-separable and faithful left $\mathcal D$ module category, then $\mathcal M$ is invertible when equipped with the $(\mathcal D,(\mathcal D^*_{\mathcal M})^{mp})$ bimodule structure.
\end{theorem}

\begin{proof}
    We will focus on the case where $\mathcal M$ is a right $\mathcal D$ module, since the dual case is analogous.
    
    1-Separability of $\mathcal M$ implies two things:
    \begin{enumerate}
        \item $\mathcal M$ is an exact module category, and
        \item By Ostrik's Theorem \cite{ostrikModuleCatsWeakHopf}, $\mathcal M\simeq {}_{A}\mathcal D$ as right module categories, for some internal, algebra $A$ in $\mathcal D$ that is separable, thanks to Theorem \ref{thm:LocSep&Sep=>Sep}.
    \end{enumerate}
    The second condition implies that $\mathcal D^*_{\mathcal M}\simeq{}_A\mathcal D_A$ as monoidal categories, where the monoidal structure on ${}_A\mathcal D_A$ is $\otimes_A$ (this is an internal version of Eilenberg-Watts).
    In particular, $\mathcal D^*_{\mathcal M}$ is 1-separable, thanks to Theorem \ref{thm:LocSep&Sep=>Sep}.
    The first condition implies that $\mathcal D^*_{\mathcal M}$ is rigid, and hence a multifusion category.
    It then follows that $\mathcal M$ is a 1-separable left $\mathcal D^*_{\mathcal M}$ module category.

    This left module structure is automatically faithful.
    Indeed, if $F$ is a nonzero summand of the identity functor, then there must be some $M\in\mathcal M$ where $F(M)\neq0$, for otherwise $F$ would be the zero functor.

    Using Ostrik's theorem again, it follows that $\mathcal M\simeq({}_A\mathcal C_A)_B$ as left ${}_A\mathcal D_A$ module categories, where $B$ is an algebra internal to ${}_A\mathcal C_A$.
    Using internal homs, this algebra can be identified with the object $B=A\otimes A^*$.
    The unit for $B$ is the map
    \[\mathbf{1}'_B:A\xrightarrow{\id\otimes\mathrm{coev_A}}A\otimes A\otimes A^*\xrightarrow{\mu_A\otimes\id}B\,.\]
    The composite
    \[\mu_B:B\otimes B\;:=\;A\otimes A^*\otimes A\otimes A^*\xrightarrow{\id\otimes\mathrm{ev}_A\otimes\id}A\otimes\1\otimes A^*\cong A\otimes A^*\;=:\;B\]
    is easily seen to be $A$-balanced, and therefore descends to a map $\mu'_B:B\otimes_A B\to B$, thus giving $B=(B,\mu'_B,\mathbf{1}'_B)$ the structure of an algebra internal to ${}_A\mathcal D_A$.
    Notice that by composing with the unit for $A$, we obtain
    \[\mathbf{1}_B:\1\xrightarrow{\mathbf{1}_A}A\xrightarrow{\mathbf{1}'_B}B\,,\]
    and the triple $(B,\mu_B,\mathbf{1}_B)$ has the structure of an algebra internal to $\mathcal D$.
    Taken together, all of these interlocking structures imply that ${}_B({}_A\mathcal D_A)_B\simeq{}_B\mathcal D_B$
    as monoidal categories.

    Consider the object $P:=A^*\otimes_BA$.  This object has the property that
    \[P\otimes P\;=\;A^*\otimes_BA\otimes A^*\otimes_BA\;=\;A^*\otimes_BB\otimes_BA\;\cong\;P\,.\;\]
    The combinatorics of fusion rules then imply that $P$ must be a summand of $\1$ in $\mathcal D$ (see, e.g. \cite[Prop. 3.4]{MR4806973}).
    If $P\subsetneq\1$, this would imply that $\mathcal M\simeq{}_A\mathcal D$ is not faithful, so we must have that $P\cong\1$.

    From here, notice that there is a canonical functor $F:\mathcal D\to{}_B\mathcal D_{B}$ given by $X\mapsto A\otimes X\otimes A^*$.
    From our observation that $A^*\otimes_BA\cong\1$, we obtain an isomorphism
    \[(A\otimes X\otimes A^*)\otimes_B(A\otimes Y\otimes A^*)\xrightarrow{\cong}A\otimes (X\otimes Y)\otimes A^*\,,\]
    and this endows $F$ with a monoidal structure.
    The map $M\mapsto A^*\otimes_BM\otimes_BA$ is easily seen be the inverse to $F$, and thus we find that $\mathcal D\simeq{}_B\mathcal D_B$.

    Finally, we can use duals to give $\mathcal M^{op}$ the structure of a $(\mathcal D,\mathcal D^*_{\mathcal M})$ bimodule, and with this structure, $\mathcal M^{op}\simeq\mathcal D_A$ as a bimodule category.
    It follows that, as bimodule categories,
    \[\mathcal M\bt_{\mathcal D}\mathcal M^{op}\simeq{}_A\mathcal D\bt_{\mathcal D}\mathcal D_A\simeq{}_A\mathcal D_A\simeq\mathcal D^*_{\mathcal M}\,.\]
    Alternatively, under the identification $\mathcal D\simeq{}_B\mathcal D_B$, $\mathcal M\simeq{}_A\mathcal D_B$ and $\mathcal M^{op}\simeq{}_B\mathcal D_A$ as bimodule categories.
    This shows that
    \[\mathcal M^{op}\bt_{\mathcal D^*_{\mathcal M}}\mathcal M\simeq {}_B\mathcal D_A\bt_{{}_A\mathcal D_A}{}_A\mathcal D_B\simeq{}_B({}_A\mathcal D_A)\bt_{{}_A\mathcal D_A}({}_A\mathcal D_A)_B\simeq{}_B({}_A\mathcal D_A)_B\simeq\mathcal D\,.\]
    Therefore, $\mathcal M$ is invertible. 
\end{proof}

\section{The category \texorpdfstring{$\mathcal C^{mp}\boxtimes_{\mathcal Z}\mathcal C$}{Cmp[x]ZC} and its action}

Let us consider the left action of $\mathcal C^{mp}\boxtimes_{\mathcal Z}\mathcal C$ on $\eta=\mathcal C$.
This comes from the natural action of $\mathcal C^{mp}\bt\mathcal C$ on $\mathcal C$ coming from left and right multiplication.
For each object $W$ in $\mathcal C$, the functor 
\begin{gather*}
    \mathcal (-\bt-)\triangleright W:C^{mp}\times\mathcal C\longrightarrow\mathcal C\\
    (X,Y)\longmapsto Y\otimes W\otimes X
\end{gather*}
admits a canonical $\mathcal Z$-balancing.
This balancing, when applied to an object $(Z,\beta^Z)$ in $\mathcal Z$, is nothing more than (the inverse of) the half-braiding.
\[
\begin{tikzcd}[ampersand replacement=\&]
	{\big[(X\triangleleft Z)\boxtimes Y\big]\triangleright W} \&\& {\big[X\boxtimes(Z\triangleright Y)\big]\triangleright W} \\
	{Y\otimes W\otimes X\otimes Z} \&\& {Z\otimes Y\otimes W\otimes X}
	\arrow["{\text{bal}^W_{X,Z,Y}}", from=1-1, to=1-3]
	\arrow[equals, from=1-1, to=2-1]
	\arrow[equals, from=1-3, to=2-3]
	\arrow["{(\beta^Z_{(Y\otimes W\otimes X)})^{-1}}", from=2-1, to=2-3]
\end{tikzcd}
\]
This behaves coherently on tensor products because of the hexagon equations for the braiding in $\mathcal Z$.
Specifically, the `1-over-2' hexagons show that this respects tensor products on $\mathcal C^{mp}$ and $\mathcal C$, while the `2-over-1' hexagons show that this respects the tensor product of objects in $\mathcal Z$.

Thus, we find that the pair $((-\bt-)\triangleright W,\text{bal}^W)$ determines a functor $(-)\blacktriangleright W:\mathcal C^{mp}\bt_{\mathcal Z}\mathcal C\to\mathcal C$, and moreover this assignment is functorial.
In other words, we have an operation
\[\blacktriangleright\in\Fun\big(\mathcal C\,,\,\Fun(\mathcal C^{mp}\bt_{\mathcal Z}\mathcal C\,,\,\mathcal C)\big)\simeq\Fun(\mathcal C^{mp}\bt_{\mathcal Z}\mathcal C\,,\,\End(\mathcal C))\,.\]
When interpreted in the latter category, $\blacktriangleright$ is monoidal, and thus determines the desired action of $\mathcal C^{mp}\bt_{\mathcal Z}\mathcal C$ on $\mathcal C$.

\begin{proposition}\label{prop:dual is Muger center}
    The module dual of $\mathcal C^{mp}\boxtimes_{\mathcal Z}\mathcal C$ with respect to its action on $\mathcal C$ is canonically equivalent to $\mathcal Z_2(\mathcal Z)$.
\end{proposition}

\begin{proof}
    Let us write $\mathcal D$ for the module dual $(\mathcal C^{mp}\bt_{\mathcal Z}\mathcal C)_{\mathcal C}^*$.
    Suppose $(F,s)$ is an object in $\mathcal D$.
    Using restriction along $Q$, we can use the module structure isomorphism $s$ for $F$ to construct the following isomorphism
    \[\begin{tikzcd}[row sep=.5,ampersand replacement=\&]
    	{F(\1)\otimes X} \& {Q(X,\1)\blacktriangleright F(\1)} \& {F\big(Q(X,\1)\blacktriangleright\1\big)} \\
    	\& {F(X)} \\
    	{F\big(Q(\1,X)\blacktriangleright\1\big)} \& {Q(\1,X)\blacktriangleright F(\1)} \& {X\otimes F(\1)}
    	\arrow["{=}", from=1-1, to=1-2]
    	\arrow["s", from=1-2, to=1-3]
    	\arrow["{=}", from=1-3, to=2-2]
    	\arrow["{=}"', from=2-2, to=3-1]
    	\arrow["{s^{-1}}", from=3-1, to=3-2]
    	\arrow["{=}", from=3-2, to=3-3]
    \end{tikzcd}\,,\]
    which is natural for $X$ in $\mathcal C$.
    This composition shows that the underlying functor of $F$ is completely determined by the object $F(\1)$.
    Let us call this natural isomorphism $\beta^{F(\1)}$.
    If this were only the action of $\mathcal C^{mp}\bt\mathcal C$ on $\mathcal C$, then we would know that
    \begin{enumerate}
        \item $(F(\1),\beta^{F(\1)})$ is an object of $\mathcal Z$, and
        \item all objects of $\mathcal Z$ arise from this construction.
    \end{enumerate}
    In our current setting, (1) is still true, but the fact that we are using $\mathcal C^{mp}\bt_{\mathcal Z}\mathcal C$ means that (2) is no longer true.
    
    The fact that this action respects the $\mathcal Z$-balancing tells us that the map $s_{-,\1}:(-)\blacktriangleright F(\1)\to F((-)\blacktriangleright\1)$ is a $\mathcal Z$-balanced natural transformation.
    This means that, for any $(Z,\beta^{Z})$ in $\mathcal Z$ and any $X,Y$ in $\mathcal C$,
    \[\begin{tikzcd}[ampersand replacement=\&]
    	{Q(X\triangleleft Z,Y)\blacktriangleright F(\1)} \&\& {F\big(Q(X\triangleleft Z,Y)\blacktriangleright\1\big)} \\
    	{Q(X,Z\triangleright Y)\blacktriangleright F(\1)} \&\& {F\big(Q(X,Z\triangleright Y)\blacktriangleright\1\big)}
    	\arrow["{s_{Q(X\triangleleft Z,Y),F(\1)}}", from=1-1, to=1-3]
    	\arrow["{\text{bal}^{F(\1)}_{X,Z,Y}}"', from=1-1, to=2-1]
    	\arrow["{F\big(\text{bal}^{\1}_{X,Z,Y}\big)}", from=1-3, to=2-3]
    	\arrow["{s_{Q(X,Z\triangleright Y),F(\1)}}"', from=2-1, to=2-3]
    \end{tikzcd}\,.\]

    Consider the following diagram, where we have specialized to $X=Y=\1$.
    \[\begin{tikzcd}[ampersand replacement=\&]
    	{Z\otimes F(\1)} \&\& {Q(\1,Z\triangleright\1)\blacktriangleright F(\1)} \\
    	{F(\1)\otimes Z} \&\& {Q(\1\triangleleft Z,\1)\blacktriangleright F(\1)} \\
    	\& {F\big(Q(\1\triangleleft Z,\1)\blacktriangleright \1\big)} \\
    	\& {F\big(Q(\1,Z\triangleright\1)\blacktriangleright \1\big)} \\
    	{Z\otimes F(\1)} \&\& {Q(\1,Z\triangleright\1)\blacktriangleright F(\1)}
    	\arrow[from=1-1, to=1-3]
    	\arrow["{\beta^Z_{F(\1)}}", from=1-1, to=2-1]
    	\arrow["{\big(\text{bal}^{F(\1)}_{\1,Z,\1}\big)^{-1}}", from=1-3, to=2-3]
    	\arrow[from=2-1, to=2-3]
    	\arrow["{\beta^{F(\1)}_Z}", from=2-1, to=5-1]
    	\arrow["s", from=2-3, to=3-2]
    	\arrow["{\text{bal}^{F(\1)}_{\1,Z,\1}}", from=2-3, to=5-3]
    	\arrow["{F\big(\text{bal}^{\1}_{\1,Z,\1}\big)}", from=3-2, to=4-2]
    	\arrow["{s^{-1}}", from=4-2, to=5-3]
    	\arrow[from=5-3, to=5-1]
    \end{tikzcd}\,.\]
    The upper rectangle commutes by the definition of $\text{bal}^{F(\1)}$.
    The lower-left hexagon commutes by the definition of $\beta^{F(\1)}$.
    The lower-right quadrangle commutes by the $\mathcal Z$-balancing property we have just considered, and so the whole diagram commutes.
    The composition of the two left vertical arrows is nothing more than the double-braiding in $\mathcal Z$ for the product $Z\otimes F(\1)$.
    Following the alternative path, we find that this double-braiding must be equal to the identity.

    This computation shows that to every $(F,s)$ in $\mathcal D$, we can associate the object $(F(\1),\beta^{F(\1)})$ in $\mathcal Z_2(\mathcal Z)$.
    It is not hard to see that any object $(T,\beta^T)$ in $\mathcal Z_2(\mathcal Z)$ will give rise to a functor $F:X\mapsto T\otimes X$, and that this is the inverse to the construction above.
\end{proof}

Now let us consider the endomorphisms of the unit in $\mathcal C^{mp}\boxtimes_{\mathcal Z}\mathcal C$.
To understand this object, we will use the calculus of ends.

Let $\Forg:\mathcal Z\to\mathcal C$ be the forgetful functor, and let $I:\mathcal C\to\mathcal Z$ be its right adjoint.
The composition $T:=\Forg\circ I$ is a comonad on $\mathcal C$, and its corresponding category of comodules is equivalent to $\mathcal Z$.
This comonad $T$ can be computed explicitly in terms of an end as follows
\[T(X):=\int_{C\in\mathcal C}C\otimes X\otimes C^*\,.\]
The end comes equipped a collection of maps \[\{w_{C;X}:T(X)\to C\otimes X\otimes C^*\}_{C\in\mathcal C}\,,\] collectively called a wedge, and maps $f:Y\to T(X)$ are completely determined by the collection $\{w_{C;X}\circ f\}$.

\begin{definition}
    The canonical central algebra associated to $\mathcal C$ is $L_{\mathcal C}:=I(\1)$.
\end{definition}

\begin{remark}
    When $\overline{\mathbb K}=\mathbb K$, the canonical central algebra is often called the canonical Lagrangian algebra, because it automatically satisfies $\FPdim(L_{\mathcal C})^2=\FPdim(\mathcal Z)$.
    This desirable property can fail if $\End(\1_{\mathcal C})\neq\mathbb K$ (compare with \cite[Thm. 4.9]{MR4806973}).
    In light of this, we choose to use the adjective central, because although it is no longer Lagrangian, it still has a canonical commutative algebra structure as an object of $\mathcal Z$.
\end{remark}

\begin{proposition}\label{prop:AltEndFormulaForL}
    If $\mathcal C$ is a 2-separable fusion category, then the canonical central algebra can be computed using the alternative formula
    \[L_{\mathcal C}\cong\int_{Z\in\mathcal Z}\mathcal C\big(\Forg(Z),\1\big)\otimes Z\,,\]
    where the tensor product of a vector space $V$ and an object $Z$ is determined by the existence of a natural isomorphism
    \[\mathcal Z(W,V\otimes Z)\cong V\otimes\mathcal Z(W,Z)\,.\]
\end{proposition}

\begin{proof}
    By semisimplicity of $\mathcal Z$, we can decompose $\mathcal L_{\mathcal C}$ into isotypic components
    \[L_{\mathcal C}\cong\bigoplus_{W\in\Irr(\mathcal Z)}W^{\oplus n_W}\,.\]
    Next, we can compute that
    \begin{align*}
        \End(Z)^{\oplus n_Z}&\cong\mathcal Z\Big(Z\;,\bigoplus_{W\in\Irr(\mathcal Z)}W^{\oplus n_W}\Big)\\
        &\cong\mathcal Z\big(Z,L_{\mathcal C}\big)\\
        &\cong\mathcal Z\big(Z,I(\1)\big)\\
        &\cong\mathcal C\big(\Forg(Z),\1\big)
    \end{align*}
    Since $\End(Z)$ is a finite dimensional separable algebra over $\Omega\mathcal Z:=\End(\1_{\mathcal Z})$, we can build a copy of $\End(Z)$ as an internal algebra whose underlying object is $\1^{\oplus k}$, for some $k$.
    From this it is immediate that
    \begin{align*}
    L_{\mathcal C}&\cong\bigoplus_{Z\in \Irr(\mathcal Z)}Z^{\oplus n_Z}\\
    &\cong\bigoplus_{Z\in \Irr(\mathcal Z)}\1^{\oplus n_Z}\otimes Z\\
    &\cong\bigoplus_{Z\in \Irr(\mathcal Z)}\End(Z)^{\oplus n_Z}\mathop{\otimes}\limits_{\End(Z)} Z\\
    &\cong\bigoplus_{Z\in \Irr(\mathcal Z)}\mathcal C\big(\Forg(Z),\1\big)\mathop{\otimes}\limits_{\End(Z)} Z\,.
    \end{align*}
    This last expression is isomorphic to the desired end.
    Normally, the defining equations for the end would carve out a maximal balanced subobject, and a coend would take the maximal balanced quotient.
    By Proposition \ref{prop:Separable<=>Z is LocSep}, our hypothesis that $\mathcal C$ is 2-separable implies that each of the algebras $\End(Z)$ is separable, and so these quotients and subobjects are the same up to isomorphism.
\end{proof}

Following \cite[Rem. 3.9]{MR2677836}, we can express $\mathcal C^{mp}\boxtimes_{\mathcal Z}\mathcal C$ as the category of modules for an internal algebra $A$ in $\mathcal C^{mp}\boxtimes\mathcal C$.
Since this algebra is built using a right adjoint, the correct formula for this algebra is
\[A:=\int_{Z\in\mathcal Z}\Forg(Z^*)\boxtimes\Forg(Z)\,.\]
In the event where $\End(Z)\cong\mathbb K$ for every simple object, this formula degenerates to the direct sum as indicated in the cited remark.

\begin{theorem}\label{thm:LoopsAlgebraIso}
    For $\mathcal C$ a 2-separable fusion category, the action of $\mathcal C^{mp}\boxtimes_{\mathcal Z}\mathcal C$ on $\mathcal C$ induces an algebra isomorphism
    \[\Omega\Big(\mathcal C^{mp}\bt_{\mathcal Z}\mathcal C\Big)\cong\End(\Id_{\mathcal C})\,.\] 
\end{theorem}

\begin{proof}
    The algebra $A$ is a commutative, central algebra in $\mathcal C^{mp}\boxtimes\mathcal C$, and thus the category of modules for $A$ inherits a monoidal product coming from $\otimes_A$.
    Under the identification that $(\mathcal C^{mp}\bt\mathcal C)_A\simeq\mathcal C^{mp}\boxtimes_{\mathcal Z}\mathcal C$, the free module functor is monoidal, and $A$ corresponds to the monoidal unit in $\mathcal C^{mp}\boxtimes_{\mathcal Z}\mathcal C$.
    A direct computation using end calculus shows that, as vector spaces,
    \begin{align*}
        \End(\Id_{\mathcal C})&\cong\int_{C\in\mathcal C}\mathcal C(C,C)\\
        &\cong\int_{C\in\mathcal C}\mathcal C(\1,C\otimes C^*)\\
        &\cong\mathcal C\Big(\1,\int_{C\in\mathcal C}C\otimes C^*\Big)\\
        &\cong\mathcal C\Big(\1,\Forg(L_{\mathcal C})\Big)\\
        &\cong\mathcal C\bigg(\1,\Forg\Big(\int_{Z\in\mathcal Z}\mathcal C\big(\Forg(Z),\1\big)\otimes Z\Big)\bigg)\\
        &\cong\mathcal C\bigg(\1,\int_{Z\in\mathcal Z}\mathcal C\big(\Forg(Z),\1\big)\otimes \Forg(Z)\bigg)\\
        &\mathop{\cong}\limits^{\star}\int_{Z\in\mathcal Z}\mathcal C\big(\Forg(Z),\1\big)\otimes\mathcal C\big(\1, \Forg(Z)\big)\\
        &\cong\int_{Z\in\mathcal Z}\mathcal C^{op,mp}\big(\Forg(Z)^*,\1^*\big)\otimes\mathcal C\big(\1, \Forg(Z)\big)\\
        &\cong\int_{Z\in\mathcal Z}\mathcal C^{mp}\big(\1,\Forg(Z^*)\big)\otimes\mathcal C\big(\1, \Forg(Z)\big)\\
        &\cong\int_{Z\in\mathcal Z}\mathcal C^{mp}\boxtimes\mathcal C\big(\;\1\boxtimes\1\,,\,\Forg(Z^*)\boxtimes\Forg(Z)\;\big)\\
        &\cong\mathcal C^{mp}\boxtimes\mathcal C\bigg(\;\1\boxtimes\1\,,\,\int_{Z\in\mathcal Z}\Forg(Z^*)\boxtimes\Forg(Z)\;\bigg)\\
        &\cong\mathcal C^{mp}\boxtimes\mathcal C\big(\;\1\boxtimes\1\,,\,A\;\big)\\
        &\cong(\mathcal C^{mp}\boxtimes\mathcal C)_A\big(A\,,\,A\big)\\
        &\cong\mathcal C^{mp}\bt_{\mathcal Z}\mathcal C\big(\1\,,\,\1\big)\,=:\Omega\big(\mathcal C^{mp}\bt_{\mathcal Z}\mathcal C\big)\;.
    \end{align*}
    Here we have used Proposition \ref{prop:AltEndFormulaForL} and the fact that the dual functor $(-)^*:\mathcal C\to\mathcal C^{op,mp}$ is an equivalence that reverses both composition and tensor product.

    \begin{remark}
        The vector space $\int_{Z\in\mathcal Z}\mathcal C\big(\Forg(Z),\1\big)\otimes\mathcal C\big(\1, \Forg(Z)\big)$ (indicated above with a $\star$) corresponds to the so-called ladder diagram formalism used in \cite{MR3975865} to compute morphism spaces in relative tensor products.
    \end{remark}

    A typical vector in $\int_{Z\in\mathcal Z}\mathcal C\big(\Forg(Z),\1\big)\otimes\mathcal C\big(\1, \Forg(Z)\big)$ looks like a sum
    \[v=\sum_{Z,i}a_i^Z\otimes b_i^Z\,.\]
    Such a vector corresponds to some $\Tilde{v}$ in $\Omega(\mathcal C^{mp}\boxtimes_{\mathcal Z}\mathcal C)$, and the action of $\Tilde{v}$ on an object $X\in\mathcal C$ is
    \[\Tilde{v}\blacktriangleright\id_X\;=\;\sum_i(\id_X\triangleleft a_i^Z)\circ\beta^Z_{X}\circ(b_i^Z\triangleright\id_X)\,.\]

    It is clear from the action formula given above that $\{f\blacktriangleright\id_X\}_{X\in\mathcal C}$ is a natural transformation of $\Id_{\mathcal C}$, and this construction respects composition because $\blacktriangleright$ is functorial.
    
    For a given natural transformation $\zeta:\Id_{\mathcal C}\to\Id_{\mathcal C}$, the maps
    \[(\zeta_X\otimes\id_{X^*})\circ\mathrm{coev}_X\]
    are $\End(X)$-balanced, and therefore determine a morphism $\zeta^\vee:\1\to L_{\mathcal C}$.

    The half-braiding on $L_{\mathcal C}$ is determined by the formula
    \[\beta^{L_{\mathcal C}}_X:=(\id_X\otimes\id_{L_{\mathcal C}}\otimes\mathrm{ev}_X)\circ(w_{X;L_{\mathcal C}}\otimes\id_X)\circ(\Delta_{\1}\otimes \id_X)\,\]
    where $\Delta:T\to T^2$ is the coproduct for the comonad $T=\Forg\circ I$, and $w_{X,L}$ is the wedge for the end that defines $T$.

    To $\zeta$, we can now associate the vector
    \[w_{\1;\1}\otimes\zeta^\vee\in\int_{Z\in\mathcal Z}\mathcal C\big(\Forg(Z),\1\big)\otimes\mathcal C\big(\1, \Forg(Z)\big)\,.\]
    The action of this vector on $\id_X$ is then
    \begin{gather*}
        (\id_X\triangleleft w_{\1;\1})\circ\beta^L_{X}\circ(\zeta^\vee\triangleright\id_X)\\
        =(\id_X\otimes w_{\1;\1}\otimes\mathrm{ev}_X)\circ \big((w_{X;L_{\mathcal C}}\circ\Delta_\1\circ\zeta^\vee)\otimes\id_X\big)\\
        =(\id_X\otimes\mathrm{ev}_X)\circ\big((w_{X;\1}\circ\zeta^\vee)\otimes\id_X\big)\\
        =(\id_X\otimes\mathrm{ev}_X)\circ(\zeta_X\otimes\id_{X^*}\otimes\id_X)\circ(\mathrm{coev}_X\otimes\id_X)\\
        =\zeta_X\,.
    \end{gather*}
    Thus we have two algebras that are (1) isomorphic as vector spaces, (2) there is an algebra map in one direction, and (3) that algebra map has a section.
    Taken together, these facts imply the result.
\end{proof}

The following consequence will be used in the proof of Theorem \ref{thm:main invertibility theorem}.

\begin{corollary}\label{cor:HardFaithfulness}
    The action of $\mathcal C^{mp}\boxtimes_{\mathcal Z}\mathcal C$ on $\mathcal C$ is faithful.
\end{corollary}

\section{A Galois theory computation and its consequences}

We are concerned with the action of $\mathcal C\bt\mathcal C^{mp}$ on $\mathcal C$, and whether or not this action is faithful.
Here we will use Galois theory to show that this faithfulness is controlled by the field $\Omega\mathcal Z\mathcal C$.

Consider a fusion category $\mathcal C$ over $\mathbb K$.
In general, $\Omega\mathcal C:=\End(\1)$ is a field extension of $\mathbb K$ which, for brevity's sake, we will denote by $\mathbb L$.
Our separability assumptions force $\mathbb L/\mathbb K$ to be a separable extension, and therefore the primitive element theorem supplies an element $\theta\in\mathbb L$ such that $\mathbb L=\mathbb K(\theta)$.

If $\mathbb L=\mathbb K$, then $\mathcal C\bt\mathcal C^{mp}$ is automatically fusion, and therefore the action is automatically faithful, so we shall assume that $\theta\notin\mathbb K$.

\begin{example}
    Consider the category $\Vec_{\mathbb C}$, thought of as a fusion category over $\mathbb K=\mathbb R$.
    The object $\1\bt\1$ in $\mathcal C\bt\mathcal C^{mp}$ decomposes as a sum of projections $\1\bt\1\cong U_1\oplus U_2$.
    This happens because
    \[\End(\1\bt\1)\cong\mathbb C\otimes_{\mathbb R}\mathbb C\;\cong\;\mathbb C\oplus\mathbb C\,.\]
    Unitality of the action implies that
    \[\1\cong\1\triangleleft(\1\bt\1)\cong (\1\triangleleft U_1)\oplus(\1\triangleleft U_2)\,.\]
    Since $\1\in\Vec_{\mathbb C}$ is simple, it follows that one of these summands must be zero, and therefore the action must not be faithful.

    This is more or less the behavior that we can expect from a Galois extension, but in general we must consider non-normal extensions as well.
\end{example}

The fact that $\mathcal C$ is 1-separable implies that $\1\bt\1\in\mathcal C\bt\mathcal C^{mp}$ decomposes as a direct sum of simple objects $U_j$.
These summands are projections in the fusion algebra of $\mathcal C\bt\mathcal C^{mp}$, and can be computed as the images of minimal idempotents in $\End(\1\bt\1)$.
To compute these idempotents, suppose $f(x)\in\mathbb K[x]$ is the minimal polynomial for $\theta$, and let $\prod_{j=1}^mf_j(x)$ be the factorization of $f(x)$ into indecomposable polynomials in $\mathbb L[x]$.
Then it follows that
\begin{align*}
    \mathbb L\otimes_{\mathbb K}\mathbb L&\cong\mathbb L\otimes_{\mathbb K}\mathbb K(\theta)\\
    &\cong L\otimes_{\mathbb K}\mathbb K[x]/\langle f(x)\rangle\\
    &\cong L[x]/\langle f(x)\rangle\\
    &\cong \prod_{j=1}^mL[x]/\langle f_j(x)\rangle\\
\end{align*}

Let $\mathbb E$ be a splitting field for $f$ over $\mathbb K$, and let $V(f):=\{\theta_i\}_{i=1}^n$ be the set of all conjugates of $\theta=\theta_1$ in $\mathbb E$.
It follows that there is a surjective function $J:\{i\}_{i=1}^n\twoheadrightarrow\{j\}_{j=1}^m$ defined by the rule that $\theta_i$ is a root of $f_{J(i)}(x)$.

\begin{remark}
    If $\mathbb L/\mathbb K$ is normal (and hence Galois), then each $f_j(x)$ has degree 1, and so $n=m$ and the map $J$ is just the identity.
    The reader should keep this case in mind for the rest of this section as a sanity check while digesting the proofs that follow.
    The necessity of complexity in the proof of Lemma \ref{lem:main faithfulness lemma} is primarily due to the existence of non-normal field extensions... which is due to the fact that groups can have non-normal subgroups... which is either happy or sad, depending on your point of view.
\end{remark}

The following polynomials will be an important tool.
\begin{definition}
    \[p_i(x):=\prod_{k\neq i}\frac{x-\theta_k}{\theta_i-\theta_k}\;\in\mathbb E[x]\]
    \[P_j(x):=\sum_{i\in J^{-1}(j)}p_i(x)\;\in\mathbb L[x]\]
\end{definition}

\begin{proposition}\label{prop:p-i and P-j}
    The following properties hold:
    \begin{enumerate}
        \item $p_i(\theta_k)=\delta_{i,k}$,
        \item $P_j(\theta_k)=\delta_{j,J(k)}$, and
        \item $\sum_{i=1}^np_i(x)\;=\;\sum_{j=1}^mP_j(x)\;=\;1$
    \end{enumerate}
\end{proposition}

\begin{proof}
    Property (1) is immediate from the defining formula for $p_i$, and property (2) follows from property (1).
    The $p_i$ are precisely the summands that are produced from Lagrange interpolation of the constant function 1, at the test points $x=\theta_i$.  Since 1 is a polynomial, Lagrange interpolation reproduces the function exactly.
    The other sum is just a regrouping of the first.
\end{proof}

Let us set $\Gamma=\Gal(\mathbb E/\mathbb K)$, and $G=\Gal(\mathbb E/\mathbb L)$.
Following Kuperberg's argument in \cite{kuperbergFiniteConnectedSemisimple2002}, we can assume that all the left embeddings $\lambda_X$ are inclusions, and we can extend all the right embeddings $\rho_X$ to elements of $\Gamma$, which by abuse of notation, we will also write as $\rho_x$.

As in the proof of \cite[Thm 4.1]{kuperbergFiniteConnectedSemisimple2002}, we can form the union of double cosets
\[H:=\bigcup_{X\in\Irr(\mathcal C)}G\rho_XG\,\]
and this $H\subseteq\Gamma$ happens to be a group, because $\mathcal C$ is a rigid monoidal category\footnote{Kuperberg's argument works more generally for bicategories, but this level of generality is not necessary for our current purposes.}.

\begin{proposition}\label{prop:fixed field}
    $\Fix(H)=\Omega\mathcal Z(\mathcal C)$.
\end{proposition}

\begin{proof}
    Let $c\in\Omega\mathcal Z(\mathcal C)\subseteq\mathbb L$, and $h\in H$ be arbitrary.
    By the definition of $H$, we can find some $X\in C$ and some $g_1,g_2\in G$ such that $h=g_1\rho_Xg_2$.
    By \cite[Prop 4.7]{MR4806973}, it follows that
    \begin{align*}
        c&=\lambda_X(c)\\
        &=\rho_X(c)\\
        &=g_1^{-1}\Big(h\big(g_2^{-1}(c)\big)\Big)\\
        &=g_1^{-1}\big(h(c)\big)\,.
    \end{align*}
    Applying $g_1$ to both sides then shows that $c\in\Fix(H)$.

    Conversely, suppose $c\in\Fix(H)$.
    Clearly $H\supseteq G$, and so $c\in\Fix(H)\subseteq\Fix(G)=\mathbb L$.
    For every simple object $X$, $\rho_X\in H$, and so $\lambda_X(c)=c=\rho_X(c)$.
    Thus, $c\in\Omega\mathcal Z(\mathcal C)$ by \cite[Prop 4.7]{MR4806973}.
\end{proof}

\begin{lemma}\label{lem:orbit decomposition}
    The group $G$ acts on $V(f):=\{\theta_i\}_{i=1}^n$, and the orbit decomposition of this $G$-set is
    \[V(f)\;=\;\coprod_{j=1}^mV(f_j)\;\;,\;\;\text{where } V(f_j):=\{\theta_i|J(i)=j\}\,.\]
    Moreover, if any two $\phi,\psi\in\Gamma$ satisfy $\phi(\theta)=\theta_i$ and $\psi(\theta)=\theta_{i'}$ with $J(i)=J(i')$, then $G\phi G=G\psi G$.
\end{lemma}

\begin{proof}
    For $1\leq i\leq n$ and $g\in G$, Proposition \ref{prop:p-i and P-j} part (2) implies that
    \[1\;=\;g(1)\;=\;g\big(P_{J(i)}(\theta_i)\big)\;=\;P_{J(i)}\big(g(\theta_i)\big)\,,\]
    where the last equality holds because the $P_j$'s have coefficients in $\mathbb L=\Fix(G)$.
    Since $g(\theta_i)=\theta_{i'}$ for some $i'$, we can apply Proposition \ref{prop:p-i and P-j}(2) in reverse to find that $J(i')=J(i)$.
    This shows that each of the sets $V(f_j):=\{\theta_i|J(i)=j\}$ is a sub $G$-set.
    Furthermore, since the $f_j(x)$ are indecomposable over $\mathbb L[x]$, it follows that the action of $G$ is transitive on $V(f_j)$ for every $j$.

    To prove the `moreover' statement, start by using transitivity of the action to find some $g\in G$ such that $g(\theta_i)=\theta_{i'}$.
    It follows that $(\psi^{-1}g\phi)(\theta)=\theta$.
    Since $\mathbb L=\mathbb K(\theta)$, we find that $\psi^{-1}g\phi$ fixes all elements of $\mathbb L$.
    In other words, there is some $g'\in G$ such that $\psi^{-1}g\phi=g'$, and this establishes the claim.
\end{proof}

\begin{lemma}\label{lem:main faithfulness lemma}
    For any 2-separable fusion category $\mathcal C$ over $\mathbb K$, the action of $\mathcal C\bt\mathcal C^{mp}$ on $\mathcal C$ is faithful if and only if $\Omega\mathcal Z(\mathcal C)=\mathbb K$.
\end{lemma}

\begin{proof}
    Our separability assumptions imply that $\mathcal C\bt\mathcal C^{mp}$ is multifusion, and therefore faithfulness of the action is equivalent to the statement that for each summand $U_j\subseteq\1\bt\1$, there exists an object $X$ such that $X\triangleleft U_j\neq0$.
    Using Proposition \ref{prop:p-i and P-j} part (2), we see that this is equivalent to the statement that for every $1\leq j\leq m$, there exists an $X$ such that $P_j\big(\rho_X(\theta)\big)=1$.
    By the Fundamental Theorem of Galois Theory, $\Fix(H)=\Fix(\Gamma)$ if and only if $H=\Gamma$.
    In light of Proposition \ref{prop:fixed field}, we find that $\Omega\mathcal Z(\mathcal C)=\mathbb K$ if and only if $H=\Gamma$.
    These equivalences are summarized in the diagram below.
    \[\begin{tikzcd}[ampersand replacement=\&]
    	{\Bigg(\;\mathcal C\curvearrowleft \mathcal C\bt\mathcal C^{mp}\;\text{ is faithful}\;\Bigg)} \&\& {\Bigg(\;\forall j,\exists X,\;P_j\big(\rho_X(\theta)\big)=1\;\bigg)} \\
    	\\
    	{\Big(\;\Omega\mathcal Z(\mathcal C)=\mathbb K\;\Big)} \&\& {\Big(\;H=\Gamma\;\Big)}
    	\arrow[Rightarrow, 2tail reversed, from=1-1, to=1-3]
    	\arrow["{\text{(iff 1)}}"{description}, Rightarrow, 2tail reversed, from=1-1, to=3-1]
    	\arrow["{\text{(iff 2)}}"{description}, Rightarrow, 2tail reversed, from=1-3, to=3-3]
    	\arrow[Rightarrow, 2tail reversed, from=3-1, to=3-3]
    \end{tikzcd}\]
    We will prove the translated equivalence (iff 2).

    For any $\gamma\in\Gamma$, there is a unique $i$ such that $\gamma(\theta)=\theta_i$.
    If we assume the action is faithful, then for $j=J(i)$, there must be some $X$ that satisfies $P_j\big(\rho_X(\theta)\big)=1$.
    By Lemma \ref{lem:orbit decomposition}, we must have $\gamma\in G\rho_XG\subseteq H$, and hence $H=\Gamma$.

    Conversely, suppose that $H=\Gamma$.
    For any $1\leq j\leq m$, choose some $i$ such that $J(i)=j$.
    Since $f(x)$ is indecomposable over $\mathbb K[x]$ and $\mathbb E$ is a splitting field for $f$, $\Gamma=\Gal(\mathbb E/\mathbb K)$ acts transitively on the roots of $f$.
    By transitivity, we can find some $\gamma\in\Gamma$ such that $\gamma(\theta)=\theta_i$.
    By assumption, we can find some simple object $X$ in $\mathcal C$ and $g_1,g_2\in G$ such that $\gamma=g_1\rho_Xg_2$.
    Now we can compute that
    \[
        1=P_j\big(\gamma(\theta)\big)
        =P_j\big((g_1\rho_Xg_2)(\theta)\big)
        =g_1\Big(P_j\big(\rho_X(\theta)\big)\Big)
        =P_j\big(\rho_X(\theta)\big)\,.
    \]
    Since $j$ was arbitrary, the proof is complete. 
\end{proof}

\begin{corollary}\label{cor:LoopsZ=K Gives an invertible bimodule}
    If $\mathcal C$ is a 2-separable fusion category over $\mathbb K$, and $\Omega\mathcal Z(\mathcal C)=\mathbb K$, then $\epsilon$ is an invertible $(\mathcal Z,\mathcal C\bt\mathcal C^{mp})$ bimodule.
\end{corollary}

\begin{proof}
    Lemma \ref{lem:main faithfulness lemma} implies that $\epsilon$ is faithful as a right module.
    Theorem \ref{thm:faithful to invertible} then tells us that $\epsilon$ is invertible as a $(\mathcal Z,\mathcal C\bt\mathcal C^{mp})$ bimodule, because Definition \ref{def:DrinfeldCenter} defines $\mathcal Z$ to be $(\mathcal C\bt\mathcal C^{mp})^*_{\mathcal C}$.
\end{proof}

\begin{corollary}\label{cor:LoopsZ=K Makes Z nondegenerate and separable}
    If $\mathcal C$ is a 2-separable fusion category over $\mathbb K$, and $\Omega\mathcal Z(\mathcal C)=\mathbb K$, then $\mathcal Z(\mathcal Z)\simeq\mathcal Z\bt\mathcal Z^{rev}$.
    In particular, this means that $\mathcal Z$ is nondegenerately braided and 2-separable.
\end{corollary}

\begin{proof}
    Corollary \ref{cor:LoopsZ=K Gives an invertible bimodule} gives invertibility of $\epsilon$, and this implies that the canonical map $\mathcal Z\bt\mathcal Z^{rev}\to\mathcal Z(\mathcal Z)$ is a braided equivalence (cf. \cite[Prop. 8.6.3]{MR3242743}), and so $\mathcal Z$ is factorizable.
    Now, we can apply \cite[Thm. 3.20]{MR4302495} to deduce that $\mathcal Z$ is nondegenerate, \emph{i.e.} $\1$ is the only simple object in $\mathcal Z_2(\mathcal Z)$.
    2-Separability of $\mathcal C$ implies that $\mathcal Z$ is 1-separable by Proposition \ref{prop:Separable<=>Z is LocSep}.
    The tensor product of 1-separable categories is again 1-separable (see Example \ref{eg:Eilenberg-Watts}), so $\mathcal Z\bt\mathcal Z^{rev}\simeq\mathcal Z(\mathcal Z)$ is 1-separable.
    Proposition \ref{prop:Separable<=>Z is LocSep} (now applied in reverse) shows that $\mathcal Z$ is 2-separable.
\end{proof}

\begin{theorem}\label{thm:main invertibility theorem}
    For $\mathcal C$ a 2-separable fusion category over $\mathbb K$, $\mathcal C$ is invertible as a 1-morphism $\mathcal Z(\mathcal C)\to\Vec_{\mathbb K}$ in $\M$ if and only if $\Omega\mathcal Z(\mathcal C)=\mathbb K$.
\end{theorem}

\begin{proof}
    If $\mathcal C$ is an invertible 1-morphism, then in particular the bimodule $\epsilon=\mathcal C$ is invertible, and hence faithful as a right $\mathcal C\bt\mathcal C^{mp}$ module category.
    Lemma \ref{lem:main faithfulness lemma} then implies that $\Omega\mathcal Z\mathcal C=\mathbb K$.

    Conversely, suppose that $\Omega\mathcal Z\mathcal C=\mathbb K$.
    By Proposition \ref{prop:inductive invertibility}, it will suffice to show that both $\epsilon$ and $\eta$ are invertible bimodule categories.
    We already have invertiblity of $\epsilon$, thanks to Corollary \ref{cor:LoopsZ=K Gives an invertible bimodule}.

    We would like to use Theorem \ref{thm:faithful to invertible} to prove that $\eta$ is also invertible.
    Since Corollary \ref{cor:HardFaithfulness} gives faithfulness of the action, it will suffice to show that the dual category is equivalent to $\Vec_{\mathbb K}$.
    By Lemma \ref{prop:dual is Muger center}, the module dual to $\mathcal C^{mp}\boxtimes_{\mathcal Z}\mathcal C$ with respect to $\eta=\mathcal C$ is $\mathcal Z_2(\mathcal Z)$ (the Müger center of the Drinfeld center), and this is just $\Vec_{\mathbb K}$ by Corollary \ref{cor:LoopsZ=K Makes Z nondegenerate and separable}.
\end{proof}

This invertibility result has several useful consequences.
Here we are content to list three.

\begin{corollary}[{cf. \cite[Thm. 2.26, part 4]{MR4302495}}]\label{cor:Cmp[x]_ZC is just End(C)}
    Suppose $\mathcal C$ is 2-separable fusion over $\mathbb K$, and $\Omega\mathcal Z\mathcal C=\mathbb K$.
    The action of $\mathcal C^{mp}\boxtimes_{\mathcal Z}\mathcal C$ on $\mathcal C$ induces an equivalence
    \[\mathcal C^{mp}\bt_{\mathcal Z}\mathcal C\to\Fun(\mathcal C,\mathcal C)\,.\]
\end{corollary}

\begin{remark}
    In the event that $\Omega\mathcal Z\mathcal C=\mathbb K$, the above result makes the statement $\Omega(\mathcal C^{mp}\boxtimes_{\mathcal Z}\mathcal C)\cong\End(\Id_{\mathcal C})$ transparent.
    Despite this clarity in hindsight, we emphasize that Theorem \ref{thm:LoopsAlgebraIso} holds in greater generality, because it does not depend on $\Omega\mathcal Z\mathcal C$.
\end{remark}

\begin{corollary}\label{cor:morita eq up to invertible}
    Suppose $\mathcal C$ and $\mathcal D$ are 2-separable fusion categories over $\mathbb K$.
    If $\mathcal Z(\mathcal C)\simeq\mathcal Z(\mathcal D)$ as braided fusion categories, and $\Omega\mathcal Z(\mathcal C)=\mathbb K$, then there exists a (Morita) invertible fusion category $\mathcal T$ such that $\mathcal D$ is Morita equivalent to $\mathcal C\bt\mathcal T$.
\end{corollary}

\begin{proof}
    This is proof is an example of what is known in physics as the folding trick.
    By Theorem \ref{thm:main invertibility theorem}, $\mathcal C$ is an invertible 1-morphism $\mathcal Z(\mathcal C)\to\Vec_{\mathbb K}$, its tensor inverse is necessarily $\mathcal C^{mp}$, thought of as an invertible 1-morphism $\Vec_{\mathbb K}\to\mathcal Z(\mathcal C)$.
    If $F:\mathcal Z(\mathcal D)\to\mathcal Z(\mathcal C)$ is a given braided equivalence, then we can use it to form another invertible 1-morphism $\mathcal Z_F:\mathcal Z(\mathcal C)\to\mathcal Z(\mathcal D)$.
    We can then form the composite
    \[\mathcal T:=\mathcal C^{mp}\bt_{\mathcal Z(\mathcal C)}\mathcal Z_F\bt_{\mathcal Z(\mathcal D)}\mathcal D\,,\]
    which is evidently an invertible 1-morphism from $\Vec_{\mathbb K}$ to itself.
    In other words, $\mathcal T$ is an invertible multifusion category.
    Invertibility implies that $\mathcal T$ is indecomposable, and so, up to composing with another Morita equivalence, we may assume that $\mathcal T$ is fusion and not multifusion.

    The invertible bimodule $\epsilon^{-1}:\mathcal Z(\mathcal C)\to\mathcal C\boxtimes\mathcal C^{mp}$, thought of as an invertible 2-morphism, can be whiskered with $\mathcal Z_F\boxtimes_{\mathcal Z(\mathcal D)}\mathcal D$ to produce a Morita equivalence from $\mathcal D$ to $\mathcal C\bt\mathcal T$.
\end{proof}

\begin{corollary}
    A 2-separable fusion category $\mathcal T$ over $\mathbb K$ is invertible if and only if $\mathcal Z(\mathcal T)=\Vec_{\mathbb K}$
\end{corollary}

\begin{proof}
    If $\mathcal Z(\mathcal T)=\Vec_{\mathbb K}$, then $\mathcal T$ satisfies the conditions of Theorem \ref{thm:main invertibility theorem}.
    Thus, $\mathcal T$ is an invertible 1-morphism from $\Vec_{\mathbb K}=\mathcal Z(\mathcal T)$ to itself, which is the same as being an invertible multifusion category.
    Since $\mathcal T$ was assumed to be fusion to begin with, it must be invertible fusion and not multifusion.
    The `only if' direction is immediate from the invertibility criteria established in \cite[Thm. 2.26]{MR4302495}.
\end{proof}

\section{The homotopy fiber of the ENO map \texorpdfstring{$\Psi$}{Ψ}}

Throughout this section, we will fix a 2-separable fusion category $\mathcal C$ over $\mathbb K$, and continue to use the shorthand $\mathcal Z:=\mathcal Z(\mathcal C)$.

In \cite{MR2677836}, for the purposes of classifying group-graded extensions of fusion categories, Etingof, Nikshych, and Ostrik construct a map $\Phi$, from the groupoid of invertible $\mathcal C$ bimodule categories $\BrPic(\mathcal C)$, to the groupoid of braided autoequivalences $\Aut_{br}\big(\mathcal Z(\mathcal C)\big)$ of the Drinfeld center.
In the algebraically closed setting, they show that under suitable truncation, this map is an equivalence of 2-groupoids.
In this section we prove a generalization of their result that works over arbitrary fields.

When working over an arbitrary field, the map $\pi_0\Phi$ can have nontrivial kernel.
This kernel can be captured homotopically by computing the homotopy fiber of $\Phi$.
It turns out that understanding the homotopy fiber also clarifies why truncation was necessary in the original result.

The automorphism $\Phi(\mathcal M)$ comes equipped with a $\mathcal C$ bimodule natural isomorphism
\[\Phi(\mathcal M)(Z)\triangleright (-)\cong (-)\triangleleft Z\,,\]
that is natural in $Z\in\mathcal Z(\mathcal C)$.
Furthermore, such an isomorphism uniquely determines $\Phi(\mathcal M)$ as a braided functor.

To construct this functor explicitly, Use invertibility of $\mathcal M$ to choose some bimodule equivalence $E:\mathcal C\to\mathcal M^{-1}\boxtimes_{\mathcal C}\mathcal C\boxtimes_{\mathcal C}\mathcal M$.
We can think of any object $Z$ in $\mathcal Z(\mathcal C)$ as a bimodule functor $Z:\mathcal C\to\mathcal C$, and under this identification, the map $\Phi$ can be defined to be
\[\Phi(\mathcal M):Z\mapsto E^{-1}\circ(\Id_{\mathcal M^{-1}}\boxtimes_{\mathcal C}Z\boxtimes_{\mathcal C}\Id_{\mathcal M})\circ E\,.\]
The functor $E$ is determined up to an invertible object $Z'$ in $\mathcal Z$.
However, both $E$ and $E^{-1}$ appear in the formula, so this ambiguity cancels out, making $\Phi(M)$ independent of this choice. 

If we think of $\mathcal C$ as a 1-morphism $\mathcal Z\to\Vec_{\mathbb K}$ in $\M$, then any invertible bimodule $\mathcal M$ induces an invertible 2-morphism
\[\mathcal M:\mathcal Z_{\Phi(\mathcal M)}\bt_{\mathcal Z}\mathcal C\to\mathcal C\,,\]
where $\mathcal Z_{\Phi(\mathcal M)}$ is the 1-morphism $\mathcal Z\to\mathcal Z$ (whose underlying monoidal category is $\mathcal Z$) determined by the braided equivalence $\Phi(\mathcal M)$.
The situation is shown diagrammatically below.
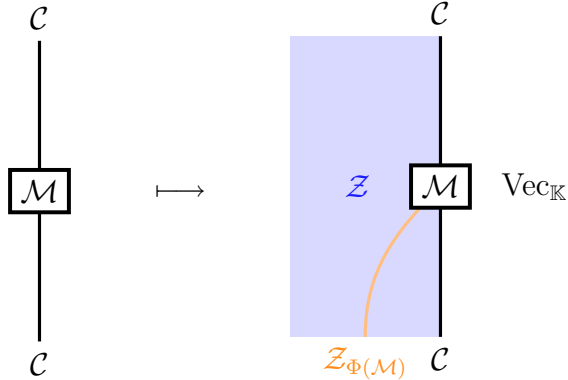
\[  
    \begin{tikzineqn}
        \draw[very thick] (0,-2) node[below]{$\mathcal C$}-- (0,0) node[draw,fill=white,line width=1.5pt]{$\mathcal M$}
            -- (0,2) node[above]{$\mathcal C$};
    \end{tikzineqn}
    \hspace{10mm}\longmapsto\hspace{10mm}
    \begin{tikzineqn}
        \fill[color=blue!15] (-2,-2) rectangle (0,2);
        \coordinate (M) at (0,0);
        \draw[very thick] (0,-2) node[below]{$\mathcal C$} -- (M)
        -- (0,2) node[above]{$\mathcal C$};
        \draw[very thick,color=orange!50] (-1,-2) to [out=90,in=225](M);
        \node[draw,fill=white,line width=1.5pt,outer sep=0] at (M) {$\mathcal M$};
        \node[color=orange!90,below] at (-1,-2) {$\mathcal Z_{\Phi(\mathcal M)}$};
        \node[color=blue!90] at (-1.1,0) {$\mathcal Z$};
        \node at (1.25,0) {$\mathrm{Vec}_{\mathbb K}$};
    \end{tikzineqn}
\]

For a bimodule equivalence $F:\mathcal M\to\mathcal N$, $\Phi$ provides us with a monoidal natural isomorphism $\Phi(F):\Phi(\mathcal M)\to\Phi(\mathcal N)$.
This can be used to produce a 2-morphism $\mathcal Z_{\Phi(F)}$ in $\M$, from $\mathcal Z_{\Phi(\mathcal M)}$ to $\mathcal Z_{\Phi(\mathcal N)}$.

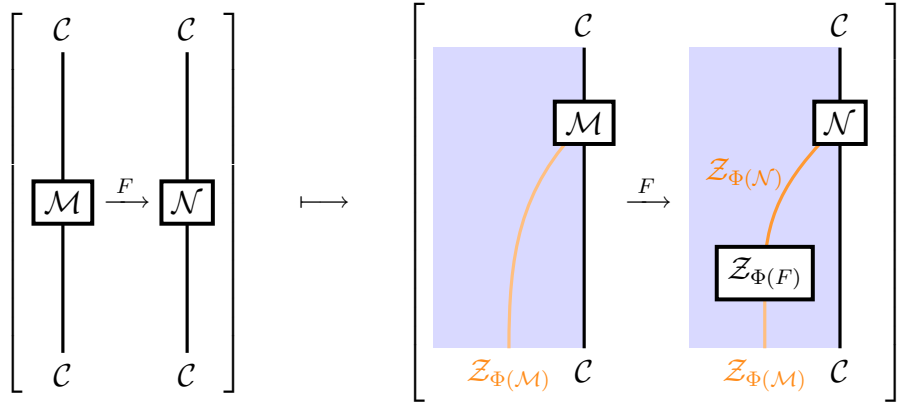
\[  
    \left[\;
        \begin{tikzineqn}
            \draw[very thick] (0,-2) node[below]{$\mathcal C$}-- (0,0) node[draw,fill=white,line width=1.5pt]{$\mathcal M$}
                -- (0,2) node[above]{$\mathcal C$};
        \end{tikzineqn}
        \xrightarrow[\hspace{4mm}]{F}
        \begin{tikzineqn}
            \draw[very thick] (0,-2) node[below]{$\mathcal C$}-- (0,0) node[draw,fill=white,line width=1.5pt]{$\mathcal N$}
                -- (0,2) node[above]{$\mathcal C$};
        \end{tikzineqn}\;
    \right]
    \hspace{6mm}\longmapsto\hspace{6mm}
    \left[\;
        \begin{tikzineqn}
            \fill[color=blue!15] (-2,-2) rectangle (0,2);
            \coordinate (M) at (0,1);
            \draw[very thick] (0,-2) node[below]{$\mathcal C$} -- (M)
            -- (0,2) node[above]{$\mathcal C$};
            \draw[very thick,color=orange!50] (-1,-2) to [out=90,in=225](M);
            \node[draw,fill=white,line width=1.5pt,outer sep=0] at (M) {$\mathcal M$};
            \node[color=orange!90,below] at (-1,-2) {$\mathcal Z_{\Phi(\mathcal M)}$};
        \end{tikzineqn}
        \xrightarrow[\hspace{4mm}]{F}\;
        \begin{tikzineqn}
            \fill[color=blue!15] (-2,-2) rectangle (0,2);
            \coordinate (M) at (0,1);
            \coordinate (F) at (-1,-1);
            \draw[very thick] (0,-2) node[below]{$\mathcal C$} -- (M)
            -- (0,2) node[above]{$\mathcal C$};
            \draw[very thick,color=orange!50] (-1,-2) to (F);
            \draw[very thick,color=orange!80] (F)
            to [out=90,in=225] node[color=orange!100!black,above left=-1.5mm]{$\mathcal Z_{\Phi(\mathcal N)}$}(M);
            \node[draw,fill=white,line width=1.5pt,outer sep=0] at (M) {$\mathcal N$};
            \node[color=orange!90,below] at (-1,-2) {$\mathcal Z_{\Phi(\mathcal M)}$};
            \node[draw,fill=white,line width=1.5pt,outer sep=0] at (F) {$\mathcal Z_{\Phi(F)}$};
        \end{tikzineqn}\;\;
    \right]
\]

\begin{definition}
    The homotopy fiber $\hofib(\Phi)$ is a 2-groupoid consisting of:
    \begin{enumerate}\setcounter{enumi}{-1}
        \item An object is a pair $(\mathcal M,\gamma^{\mathcal M})$ of an object $\mathcal M$ in $\BrPic(\mathcal C)$ and a monoidal natural isomorphism $\gamma^{\mathcal M}:\id_{\mathcal Z}\rightarrow\Phi(\mathcal M)$.
        \item A 1-morphism $F:(\mathcal M,\gamma^{\mathcal M})\to(\mathcal N,\gamma^{\mathcal N})$ is a bimodule equivalence $F:\mathcal M\to\mathcal N$, such that $\Phi(F)\circ\gamma^{\mathcal M}=\gamma^{\mathcal N}$ as (monoidal) natural isomorphisms.
        \item A 2-morphism $\nu:F\to G$ is a bimodule natural isomorphism $\nu:F\to G$.
    \end{enumerate}
    The composition of bimodules induces a monoidal structure on $\hofib(\Phi)$ by the rule
    \begin{gather}
    (\mathcal M,\gamma^{\mathcal M})\bt_{\mathcal C}(\mathcal N,\gamma^{\mathcal N}):=\big(\mathcal M\bt_{\mathcal C}\mathcal N\,,\,\gamma^{\mathcal M}*\gamma^{\mathcal N}\big)\;,\label{eqn:tensor product in hofib}\\
    \gamma^{\mathcal M}*\gamma^{\mathcal N}:=\Phi^{(2)}_{\mathcal M,\mathcal N}\circ(\gamma^{\mathcal M}\gamma^{\mathcal N})\;,\label{eqn:convolution of gammas}
    \end{gather}
    where $\Phi^{(2)}_{\mathcal M,\mathcal N}:\Phi(\mathcal M)\Phi(\mathcal N)\to\Phi(\mathcal M\bt_{\mathcal C}\mathcal N)$ is the monoidal structure morphism for $\Phi$.
\end{definition}

For an object $(\mathcal M,\gamma^{\mathcal M})$ in $\hofib(\Phi)$, the map $\gamma^{\mathcal M}$ can be used to construct a 2-morphism $\mathcal Z_{\gamma^\mathcal M}:\id_{\mathcal Z}\to\mathcal Z_{\Phi(\mathcal M)}$ in $\M$.
Whiskering $\mathcal Z_{\gamma^\mathcal M}$ with the identity bimodule on $\mathcal C$ gives us a 2-morphism in $\M$ that we will call $\Triv(\mathcal M,\gamma^{\mathcal M})$.
This construction is depicted below.
\begin{equation}\label{eqn:definition of the map A}
    \begin{tikzineqn}
        \fill[color=blue!15] (-2.5,-2) rectangle (0,1);
        \coordinate (M) at (0,-.5);
        \draw[very thick] (0,-2) node[below]{$\mathcal C$} -- (M)
        -- (0,1) node[above]{$\mathcal C$};
        \node[draw,fill=white,line width=1.5pt,outer sep=0] at (M) {$\Triv(\mathcal M,\gamma^{\mathcal M})$};
        \node[color=blue!90] at (-1.9,-.5) {$\mathcal Z$};
        \node at (2.25,-0.5) {$\mathrm{Vec}_{\mathbb K}$};
    \end{tikzineqn}
    \hspace{3mm}:=\hspace{4mm}
    \begin{tikzineqn}
        \fill[color=blue!15] (-2,-2) rectangle (0,1);
        \coordinate (M) at (0,0);
        \coordinate (G) at (-1,-1);
        \draw[very thick] (0,-2) node[below]{$\mathcal C$} -- (M)
        -- (0,1) node[above]{$\mathcal C$};
        \draw[very thick,color=orange!50] (G) to [out=90,in=225](M);
        \node[draw,fill=white,line width=1.5pt,outer sep=0] at (G) {$\mathcal Z_{\gamma^\mathcal M}$};
        \node[draw,fill=white,line width=1.5pt,outer sep=0] at (M) {$\mathcal M$};
        \node[color=blue!90] at (-1.1,0) {$\mathcal Z$};
        \node at (1.25,-0.5) {$\mathrm{Vec}_{\mathbb K}$};
    \end{tikzineqn}
\end{equation}
Since every factor present in the above composition is invertible, this is an auto-2-morphism of the 1-morphism $\mathcal C$ in $\M$.
In other words, this construction provides a map $\Triv:\hofib(\Phi)\to\Aut_{\M}(\mathcal C)$.
The notation $\Triv$ is meant to evoke the idea that this construction uses the given isomorphism $\gamma^{\mathcal M}$ to `trivialize' the defect $\mathcal Z_{\Phi(\mathcal M)}$.

Let us analyze the effect of $\Triv$ on morphisms.
Given $F:(\mathcal M,\gamma^{\mathcal M})\to(\mathcal N,\gamma^{\mathcal N})$, the equality $\Phi(F)\circ\gamma^{\mathcal M}=\gamma^{\mathcal N}$ provides a canonical equivalence $\mathcal Z_{[F]}:\mathcal Z_{\Phi(F)}\boxtimes_{\mathcal Z_{\Phi(\mathcal M)}}\mathcal Z_{\gamma^{\mathcal M}}\to\mathcal Z_{\gamma^{\mathcal N}}$.
The value of $\Triv(F)$ is then the composition, depicted below, of $F$ followed by $\mathcal Z_{[F]}$.
\[
    \begin{tikzineqn}
        \fill[color=blue!15] (-2.1,-2) rectangle (0,2);
        \coordinate (M) at (0,1.5);
        \coordinate (F1) at (-1,-1.4);
        \draw[very thick] (0,-2) node[below]{$\mathcal C$} -- (M)
        -- (0,2) node[above]{$\mathcal C$};
        \draw[very thick,color=orange!50] (F1) to[out=90,in=225] node[color=orange!70,left=-1mm] {$\mathcal Z_{\Phi(\mathcal M)}$}(M);
        \node[draw,fill=white,line width=1.5pt,outer sep=0] at (M) {$\mathcal M$};
        \node[draw,fill=white,line width=1.5pt,outer sep=0] at (F1) {$\mathcal Z_{\gamma^{\mathcal M}}$};
    \end{tikzineqn}
    \raisebox{9mm}{$\xrightarrow[\hspace{4mm}]{F}$}\;
    \begin{tikzineqn}
        \fill[color=blue!15] (-2.4,-2) rectangle (0,2);
        \coordinate (M) at (0,1.5);
        \coordinate (F1) at (-1,-1.4);
        \coordinate (F2) at (-1,0);
        \draw[very thick] (0,-2) node[below]{$\mathcal C$} -- (M)
        -- (0,2) node[above]{$\mathcal C$};
        \draw[very thick,color=orange!50] (F1) to node[color=orange!70,left] {$\mathcal Z_{\Phi(\mathcal M)}$} (F2);
        \draw[very thick,color=orange!80] (F2)
        to [out=90,in=225] node[color=orange!100!black,above left=-1.5mm]{$\mathcal Z_{\Phi(\mathcal N)}$}(M);
        \node[draw,fill=white,line width=1.5pt,outer sep=0] at (M) {$\mathcal N$};
        \node[draw,fill=white,line width=1.5pt,outer sep=0] at (F2) {$\mathcal Z_{\Phi(F)}$};
        \node[draw,fill=white,line width=1.5pt,outer sep=0] at (F1) {$\mathcal Z_{\gamma^{\mathcal M}}$};
    \end{tikzineqn}
    \raisebox{-8mm}{$\xrightarrow[\hspace{4mm}]{\mathcal Z_{[F]}}$}\;
    \begin{tikzineqn}
        \fill[color=blue!15] (-2,-2) rectangle (0,2);
        \coordinate (M) at (0,1.5);
        \coordinate (F1) at (-1,-.8);
        \draw[very thick] (0,-2) node[below]{$\mathcal C$} -- (M)
        -- (0,2) node[above]{$\mathcal C$};
        \draw[very thick,color=orange!90] (F1) to[out=90,in=225] node[color=orange!100!black,above left=-1.5mm]{$\mathcal Z_{\Phi(\mathcal N)}$} (M);
        \node[draw,fill=white,line width=1.5pt,outer sep=0] at (M) {$\mathcal N$};
        \node[draw,fill=white,line width=1.5pt,outer sep=0] at (F1) {$\mathcal Z_{\gamma^{\mathcal N}}$};
    \end{tikzineqn}
\]

For a 2-morphism $\nu:F\to G$, it follows that $\Phi(F)=\Phi(G)$, and the image $\Triv(\nu)$ is just $\nu:F\to G$ whiskered with $\mathcal Z_{[F]}=\mathcal Z_{[G]}$.

Finally, let us consider the monoidal structure of $\Triv:\hofib(\Phi)\to\Aut_{\M}(\mathcal C)$.
We demonstrate this structure map for $\Triv$ diagrammatically below, because in-line formulas would be more cumbersome and hard to parse.
The key point is that, given any two braided functors $\Psi,\Theta\in\Aut_{br}(\mathcal Z)$, there is a canonical, $\mathcal Z$-central, monoidal equivalence $\mathcal Z_\Psi\boxtimes_{\mathcal Z}\mathcal Z_\Theta\simeq\mathcal Z_{\Psi\Theta}$, which we denote as an unmarked trivalent vertex in the diagram below.

\begin{gather*}
    \begin{tikzineqn}
        \fill[color=blue!15] (-2,-3) rectangle (0,3);
        \coordinate (M) at (0,2);
        \coordinate (GM) at (-1,1);
        \coordinate (N) at (0,-1);
        \coordinate (GN) at (-1,-2);
        \draw[very thick] (0,-3) node[below]{$\mathcal C$} -- (M)
        -- (N)
        -- (0,3) node[above]{$\mathcal C$};
        \draw[very thick,color=orange!50] (GM) to [out=90,in=225](M);
        \draw[very thick,color=orange!90] (GN) to [out=90,in=225](N);
        \node[draw,fill=white,line width=1.5pt,outer sep=0] at (GM) {$\mathcal Z_{\gamma^\mathcal M}$};
        \node[draw,fill=white,line width=1.5pt,outer sep=0] at (GN) {$\mathcal Z_{\gamma^\mathcal N}$};
        \node[draw,fill=white,line width=1.5pt,outer sep=0] at (M) {$\mathcal M$};
        \node[draw,fill=white,line width=1.5pt,outer sep=0] at (N) {$\mathcal N$};
    \end{tikzineqn}
    \;\simeq\;
    \begin{tikzineqn}
        \fill[color=blue!15] (-3,-3) rectangle (0,3);
        \coordinate (M) at (0,2);
        \coordinate (GM) at (-2.1,-2);
        \coordinate (N) at (0,1);
        \coordinate (GN) at (-.9,-2);
        \draw[very thick] (0,-3) node[below]{$\mathcal C$} -- (M)
        -- (N)
        -- (0,3) node[above]{$\mathcal C$};
        \draw[very thick,color=orange!50] (GM) to [out=90,in=225](M);
        \draw[very thick,color=orange!90] (GN) to [out=90,in=225](N);
        \node[draw,fill=white,line width=1.5pt,outer sep=0] at (GM) {$\mathcal Z_{\gamma^\mathcal M}$};
        \node[draw,fill=white,line width=1.5pt,outer sep=0] at (GN) {$\mathcal Z_{\gamma^\mathcal N}$};
        \node[draw,fill=white,line width=1.5pt,outer sep=0] at (M) {$\mathcal M$};
        \node[draw,fill=white,line width=1.5pt,outer sep=0] at (N) {$\mathcal N$};
    \end{tikzineqn}
    \;\simeq\;
    \begin{tikzineqn}
        \fill[color=blue!15] (-3,-3) rectangle (0,3);
        \coordinate (M) at (0,2);
        \coordinate (GM) at (-2.1,-2);
        \coordinate (N) at (0,1);
        \coordinate (GN) at (-.9,-2);
        \coordinate (U) at (-1.5,-.25);
        \coordinate (L) at (-1.5,-.75);
        \draw[very thick] (0,-3) node[below]{$\mathcal C$} -- (M)
        -- (N)
        -- (0,3) node[above]{$\mathcal C$};
        \draw[very thick,color=orange!50] (GM) to [out=90,in=225](L);
        \draw[very thick,color=orange!90] (GN) to [out=90,in=315](L);
        \draw[very thick,color=brown!90] (L) to (U);
        \draw[very thick,color=orange!50] (U) to [out=135,in=225](M);
        \draw[very thick,color=orange!90] (U) to [out=45,in=225](N);
        \node[draw,fill=white,line width=1.5pt,outer sep=0] at (GM) {$\mathcal Z_{\gamma^\mathcal M}$};
        \node[draw,fill=white,line width=1.5pt,outer sep=0] at (GN) {$\mathcal Z_{\gamma^\mathcal N}$};
        \node[draw,fill=white,line width=1.5pt,outer sep=0] at (M) {$\mathcal M$};
        \node[draw,fill=white,line width=1.5pt,outer sep=0] at (N) {$\mathcal N$};
        \draw[very thick,dotted,rounded corners,color=gray] (-2,-.5) rectangle (.75,2.5);
    \end{tikzineqn}\\
    \;=\;
    \begin{tikzineqn}
        \fill[color=blue!15] (-3,-3) rectangle (0,3);
        \coordinate (MN) at (0,1.75);
        \coordinate (GM) at (-2.1,-2);
        \coordinate (GN) at (-.9,-2);
        \coordinate (U) at (-1.5,-.25);
        \coordinate (L) at (-1.5,-.75);
        \draw[very thick] (0,-3) node[below]{$\mathcal C$} -- (MN)
        -- (0,3) node[above]{$\mathcal C$};
        \draw[very thick,color=orange!50] (GM) to [out=90,in=225](L);
        \draw[very thick,color=orange!90] (GN) to [out=90,in=315](L);
        \draw[very thick,color=brown!90] (L) to (U);
        \draw[very thick,color=brown!90] (U) to [out=90,in=225](MN);
        \node[draw,fill=white,line width=1.5pt,outer sep=0] at (GM) {$\mathcal Z_{\gamma^\mathcal M}$};
        \node[draw,fill=white,line width=1.5pt,outer sep=0] at (GN) {$\mathcal Z_{\gamma^\mathcal N}$};
        \node[draw,fill=white,line width=1.5pt,outer sep=0] at (MN) {$\mathcal M\bt_{\mathcal C}\mathcal N$};
        \draw[very thick,dotted,rounded corners,color=gray] (-2.875,-2.75) rectangle (-.125,-.5);
    \end{tikzineqn}
    \;=\;
    \begin{tikzineqn}
        \fill[color=blue!15] (-3,-3) rectangle (0,3);
        \coordinate (MN) at (0,1.75);
        \coordinate (GMN) at (-1.5,-1.5);
        \draw[very thick] (0,-3) node[below]{$\mathcal C$} -- (MN)
        -- (0,3) node[above]{$\mathcal C$};
        \draw[very thick,color=brown!90] (GMN) to [out=90,in=225](MN);
        \node[draw,fill=white,line width=1.5pt,outer sep=0] at (GMN) {$\mathcal Z_{(\gamma^{\mathcal M}*\gamma^{\mathcal N})}$};
        \node[draw,fill=white,line width=1.5pt,outer sep=0] at (MN) {$\mathcal M\bt_{\mathcal C}\mathcal N$};
    \end{tikzineqn}
\end{gather*}

\begin{proposition}\label{prop:Triv is an equivalence}
    The map $\Triv:\hofib(\Phi)\to\Aut_{\M}(\mathcal C)$ is an equivalence of monoidal 2-groupoids.
\end{proposition}

\begin{proof}    
    Our notation assumes that $\mathcal C$ is being thought of as a 1-morphism $\mathcal Z\to\Vec_{\mathbb K}$.
    In particular, this implies that any $\mathcal M'\in\Aut_{\M}(\mathcal C)$ must be $\mathcal Z$-central, and therefore must satisfy $\Phi(\mathcal M')=\id_{\mathcal Z}$.
    Thus, we can construct an object $\Glor(\mathcal \mathcal M'):=(\mathcal M',\id_{\id_{\mathcal Z}})\in\hofib(\Phi)$.
    Since this construction will be an inverse to $\Triv=$ trivialization, we will call it $\Glor=$ glorification, with the idea being that we are overselling the significance of $\mathcal M'$ by bundling it up with trivial trivialization data.

    Immediately from the definitions, it is clear that $\Triv\Glor(\mathcal M')=\mathcal M'$.
    What about $\Glor\Triv(\mathcal M,\gamma^{\mathcal M})$?
    We are looking to build an equivalence $F:(\Triv(\mathcal M,\gamma^{\mathcal M}),\id_{\id_{\mathcal Z}})\to(\mathcal M,\gamma^{\mathcal M})$.
    Crucially, the underlying equivalence $F:\Triv(\mathcal M,\gamma^{\mathcal M})\to\mathcal M$ is not required to be $\mathcal Z$-central, it only needs to be a $\mathcal C$ bimodule equivalence.
    Since we can ignore $\mathcal Z$-centrality data for the moment, we find that $\mathcal Z_{\Phi(\mathcal M)}\simeq\mathcal Z$ as monoidal categories, and $\mathcal Z_{\gamma^{\mathcal M}}\simeq\mathcal Z$ as left $\mathcal Z$ module categories.  Therefore,
    \begin{gather*}
        \Triv(\mathcal M,\gamma^{\mathcal M}):=\mathcal M\bt_{\mathcal Z_{\Phi(\mathcal M)}\bt_{\mathcal Z}\mathcal C}(\mathcal Z_{\gamma^{\mathcal M}}\bt_{\mathcal Z}\id_{\mathcal C})
        \simeq\mathcal M\bt_{\mathcal C}\mathcal C\simeq\mathcal M\,
    \end{gather*}
    as $\mathcal C$ bimodule categories.
    
    Let us use the above equivalence as our $F$.
    Since composing with $\id_{\id_{\mathcal Z}}$ does nothing, we would need $\Phi(F)=\gamma^{\mathcal M}$ in order for this to be an actual 1-morphism in $\hofib(\Phi)$.
    This is where the $\mathcal Z$-centrality data comes back into play.
    For $Z\in \mathcal Z$, the isomorphism $\Phi(F)_Z$ is determined by the requirement that the following diagram commutes for all $T\in\Triv(\mathcal M,\gamma^{\mathcal M})$:
    \[\begin{tikzcd}[ampersand replacement=\&]
    	{Z\triangleright F(T)} \&\& {\Phi(\mathcal M)(Z)\triangleright F(T)} \\
    	{F(Z\triangleright T)} \& {F(T\triangleleft Z)} \& {F(T)\triangleleft Z}
    	\arrow["{\Phi(F)_Z\triangleright\id_{F(T)}}", from=1-1, to=1-3]
    	\arrow[from=1-1, to=2-1]
    	\arrow[from=2-1, to=2-2]
    	\arrow[from=2-2, to=2-3]
    	\arrow[from=2-3, to=1-3]
    \end{tikzcd}\;\;.\]

    These morphisms can be interpreted inside $\M$ using the $\mathcal Z$-central structures on all the modules.
    We can visualize the path of the object $Z$ with the following diagram.

    
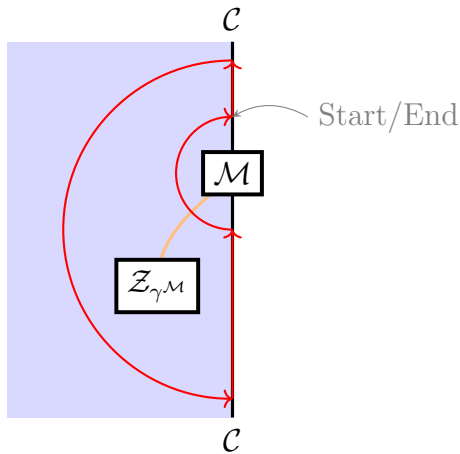
\[
        \begin{tikzpicture}
            \fill[color=blue!15] (-3,-3) rectangle (0,2);
            \coordinate (M) at (0,.25);
            \coordinate (G) at (-1,-1.25);
            \draw[very thick] (0,-3) node[below]{$\mathcal C$} -- (M)
            -- (0,2) node[above]{$\mathcal C$};
            \draw[very thick,color=orange!50] (G) to [out=90,in=225](M);
            \node[draw,fill=white,line width=1.5pt,outer sep=0] at (G) {$\mathcal Z_{\gamma^\mathcal M}$};
            \node[draw,fill=white,line width=1.5pt,outer sep=0] at (M) {$\mathcal M$};
            \draw[->,thick,color=red] (0,1) -- (0,1.75);
            \draw[->,thick,color=red] (0,1.75) arc (90:270:2.25cm);
            \draw[->,thick,color=red] (0,-2.75) -- (0,-.5);
            \draw[->,thick,color=red] (0,-.5) arc (270:90:7.5mm);
            arc (90:270:22mm);
            edge[->] (0,.5);
            \draw[-{stealth},color=gray,bend right] (1,1) node[right]{Start/End} to (0,1);
        \end{tikzpicture}
    \]
    The map $\Phi(F)_Z:Z\to\Phi(\mathcal M)(Z)$ is completely determined by the homotopy class of the red path shown above.
    Furthermore, the $\mathcal Z$-central structure on $\mathcal C$ allows the path to detach from the boundary via homotopy.
    This shows that $\Phi(F)_Z$ is determined by the fact that this path encircles the module $\mathcal Z_{\gamma^{\mathcal M}}$.
    According to the $\mathcal Z$-central structure used to define this module category, this map is precisely $\gamma^{\mathcal M}_Z$.

    Thus $\Glor\Triv\simeq\id$ and $\Triv\Glor\simeq\id$, so this forms an equivalence of 2-groupoids.
    Since $\Triv$ was monoidal to begin with, this is an equivalence of monoidal 2-groupoids.
\end{proof}

\begin{theorem}\label{thm:fiber sequence}
    Suppose that $\mathcal C$ is 2-separable fusion over $\mathbb K$ and $\Omega\mathcal Z(\mathcal C)=\mathbb K$.  There is a homotopy fiber sequence
    \begin{equation*}\label{eqn:The fibration}
    \BrPic(\Vec_{\mathbb K})\to\BrPic(\mathcal C)\xrightarrow{\Phi}\Aut_{br}\big(\mathcal Z(\mathcal C)\big)\,.
    \end{equation*}
\end{theorem}

\begin{proof}
    The defining property of the homotopy fiber is that it provides a fiber sequence of the form:
    \begin{equation*}\label{eqn:tautological fiber sequence}
        \hofib(\Phi)\to\BrPic(\mathcal C)\xrightarrow{\Phi}\Aut_{br}\big(\mathcal Z(\mathcal C)\big)\,.
    \end{equation*}
    Proposition \ref{prop:Triv is an equivalence} shows that we can replace $\hofib(\Phi)$ with $\Aut_{\M}(\mathcal C)$.
    By Theorem \ref{thm:main invertibility theorem}, $\mathcal C:\mathcal Z(\mathcal C)\to\Vec_{\mathbb K}$ is an invertible 1-morphism in $\M$.
    It follows that composition on the left with $\mathcal C$ induces a further equivalence of monoidal 2-groupoids
    \[\mathcal C\bt(-):\Aut_{\M}(\id_{\Vec_{\mathbb K}})\to\Aut_{\M}(\mathcal C)\,.\]
    
    Finally, since $\M$ is defined over $\mathbb K$, $\Vec_{\mathbb K}$-central structures are always trivial.
    This means that $\Aut_{\M}(\id_{\Vec_{\mathbb K}})$ consists of invertible bimodules for $\Vec_{\mathbb K}$, bimodule autoequivalences, and bimodule natural isomorphisms, and this is precisely $\BrPic(\Vec_{\mathbb K})$.
\end{proof}

\begin{corollary}\label{cor:5-term LES}
    Suppose that $\mathcal C$ is 2-separable fusion over $\mathbb K$, and $\Omega\mathcal Z(\mathcal C)=\mathbb K$.
    Then there is a 5-term left exact sequence
    \[\Inv\big(\mathcal Z(\mathcal C)\big)\hookrightarrow\aut_{\otimes}(\Id_{\mathcal Z(\mathcal C)})\to\Br(\mathbb K)\to\brpic(\mathcal C)\to\aut_{br}\big(\mathcal Z(\mathcal C)\big)\,.\]
\end{corollary}

\begin{proof}
    This follows from the long exact sequence of the fibration from Theorem \ref{thm:fiber sequence}.
    It was already observed in \cite[Prop. 4.9]{MR2677836} that $\pi_0(\BrPic(\Vec_{\mathbb K}))=\brpic(\Vec_{\mathbb K})=\Br(\mathbb K)$.
    The remaining homotopy groups were identified in \cite[Prop. 7.1]{MR2677836}.
    In particular, $\pi_1(\BrPic(\Vec_{\mathbb K}))\cong\Inv(\mathcal Z(\Vec_{\mathbb K}))=\Inv(\Vec_{\mathbb K})=0$.
    This explains why the first map in the 5-term sequence is injective, and also why this portion of the full long exact sequence breaks off into these last 5 terms.
\end{proof}

The lack of surjectivity at the end of the sequence in Corollary \ref{cor:5-term LES} makes our result logically incomparable to \cite[Thm. 1.1]{MR2677836}: It is stronger in one way, and weaker in another.
Lack of surjectivity comes from the existence of nontrivial invertible fusion categories that were discovered in \cite{sanford2024invertiblefusioncategories}.
To see this, let us establish a construction.

\begin{definition}\label{def:the map T}
    Given an $F\in\Aut_{br}(\mathcal Z(\mathcal C))$, let $\mathcal T_F$ be the following composition of 1-morphisms in $\M$:
    \begin{equation}
        \mathcal T_F:=\mathcal C^{mp}\bt_{\mathcal Z}\mathcal Z_F\bt_{\mathcal Z}\mathcal C\,.
    \end{equation}
    If $\mathcal C$ is invertible in $\M$, then $\mathcal T_F$ will be invertible, and this construction determines a homomorphism
    \[\mathcal T:\aut_{br}\big(\mathcal Z(\mathcal C)\big)\to\aut_{\M}(\Vec_{\mathbb K})\,\]
\end{definition}

Using this map $\mathcal T$, we can extend the exact sequence one term further.
The group $\aut_{\M}(\Vec_{\mathbb K})$ is the group of all invertible fusion categories over $\mathbb K$, up to Morita equivalence.
This group was computed in \cite[Thm. 5.9]{sanford2024invertiblefusioncategories} to be isomorphic to $H^3(\mathbb K;\mathbb G_m)$.

\begin{theorem}\label{thm:what's the cokernel?}
    Suppose that $\mathcal C$ is 2-separable fusion over $\mathbb K$, and $\Omega\mathcal Z(\mathcal C)=\mathbb K$.
    The long exact sequence of homotopy groups associated to the fibration of Theorem \ref{thm:fiber sequence} can be extended by appending $\mathcal T$ (shown below), though $\mathcal T$ itself is not necessarily surjective.
    \[
    \begin{tikzpicture}[x=3cm,y=1.2cm]
        \node (A) at (0,2) {$\mathbb K^\times$};
        \node (B) at (1,2) {$\mathbb K^\times$};
        \node (C) at (2,2) {$0$};
        \node (D) at (0,1) {$0$};
        \node (E) at (1,1) {$\mathrm{Inv}\big(\mathcal Z(\mathcal C)\big)$};
        \node (F) at (2,1) {$\mathrm{Aut}_{\otimes}(\mathrm{Id}_{\mathcal Z(\mathcal C)})$};
        \node (G) at (0,0) {$\mathrm{Br}(\mathbb K)$};
        \node (H) at (1,0) {$\mathrm{BrPic}(\mathcal C)$};
        \node (I) at (2,0) {$\mathrm{Aut}_{br}\big(\mathcal Z(\mathcal C)\big)$};
        \node (J) at (0,-1) {$H^3(\mathbb K;\mathbb G_m)$};
        \draw[->, rounded corners]
        (A) edge (B)
        (B) edge (C)
        (C) -- ++(0.6,0) -- ++(0,-.5) -- ++(-3.1,0) -- ++(0,-.5) -- ++(.1,0) edge (D)
        (D) edge (E)
        (E) edge  (F)
        (F) -- ++(0.6,0) -- ++(0,-.4) -- ++(-3.1,0) -- ++(0,-.6) -- ++(.1,0) edge (G)
        (G) edge (H)
        (H) -- node[above]{$\pi_0\Phi$} (I);
        \draw[->, rounded corners, orange] (I) -- ++(0.6,0) -- ++(0,-.5) -- node[below]{$\mathcal T$} ++(-3.1,0) -- ++(0,-.5) -- (J);
    \end{tikzpicture}
    \]
\end{theorem}

As remarked in the intro, the full picture of this long exact sequence makes the role of Galois cohomology evident, as the left column is precisely $H^*(\mathbb K;\mathbb G_m)$.

\begin{proof}
    The only thing that remains to be shown is that $\im(\pi_0\Phi)=\ker(\mathcal T)$.
    This argument is essentially just the folding trick again.
    In other words, it is just a bending and unbending of an adjunction, and is thus straightforward.
    We encourage the reader to supply the pictures that correspond to the relevant composite bimodules themselves.
    
    First, suppose that $F=\Phi(\mathcal M)$.
    By Theorem \ref{thm:main invertibility theorem}, $\mathcal C$ is invertible, and so there is a $\mathcal Z$-central invertible bimodule $\eta^{-1}$ from $\mathcal C^{mp}\boxtimes_{\mathcal Z}\mathcal C$ to $\Vec_{\mathbb K}$.
    We can then form the composite
    \[\eta^{-1}\bt_{\mathcal C^{mp}\bt_{\mathcal Z}\mathcal C}(\mathcal C^{mp}\bt_{\mathcal Z}\mathcal M)\,.\]
    This is an invertible bimodule from $\mathcal T_F$ to (the identity 1-morphism of) $\Vec_{\mathbb K}$, so $\mathcal T_F$ must represent the trivial element in $\aut_{\M}(\Vec_{\mathbb K})$.

    Conversely, if $\mathcal T_F$  represents the trivial element in $\aut_{\M}(\Vec_{\mathbb K})$, then we can choose some invertible bimodule $\mathcal N$ from $\mathcal T_F$ to $\Vec_{\mathbb K}$.
    Invertibility of $\mathcal C$ implies that we can find some invertible bimodule $\epsilon^{-1}$ from $\mathcal Z$ to $\mathcal C\boxtimes\mathcal C^{mp}$.
    Next, form the composite
    \[\mathcal M:=(\mathcal C\boxtimes\mathcal N)\bt_{\mathcal C\boxtimes\mathcal C^{mp}\bt_{\mathcal Z}\mathcal C}(\,\epsilon^{-1}\,\bt_{\mathcal Z}\mathcal Z_F\bt_{\mathcal Z}\mathcal C)\,.\]
    This $\mathcal M$ is an invertible 2-morphism from $\mathcal Z_F\bt_{\mathcal Z}\mathcal C$ to $\mathcal C$, and therefore $\Phi(\mathcal M)=F$.
\end{proof}

\begin{corollary}\label{cor:H^3 condition}
    Suppose that $\mathcal C$ is 2-separable fusion over $\mathbb K$ and $\Omega\mathcal Z(\mathcal C)=\mathbb K$.
    The map $\pi_0(\Phi)$ is surjective if and only if $\mathcal T$ is the zero map.
    In particular, if $H^3(\mathbb K;\mathbb G_m)=1$, then $\pi_0(\Phi)$ is surjective.
\end{corollary}

\begin{remark}[cf. {\cite[Eg.s 2.8 \& 2.9]{sanford2024invertiblefusioncategories}}]\label{rem:good and bad fields}
    The property $H^3(\mathbb K;\mathbb G_m)=1$ holds when $\mathbb K$ is algebraically closed, local, or global.
    The rational function field $\mathbb C(x,y,z)$ is an example of a field where $H^3$ is nontrivial.
\end{remark}

At the moment of writing, we do not know of any examples where $\mathcal T$ is nonzero.
Regarding this issue, we pause a moment to pose a few questions.

\begin{question}
    Do there exist any fields $\mathbb K$, and fusion categories $\mathcal C$ over $\mathbb K$, such that $\im(\mathcal T)\neq0$?
\end{question}

If the answer is yes, then...

\begin{question}
    Given $\mathcal C$ over $\mathbb K$, and $F\in\Aut_{br}(\mathcal Z(\mathcal C))$, is there a method that allows for the computation of the cocycle $\omega\in H^3(\mathbb K;\mathbb G_m)$ corresponding to $\mathcal T_F$?
\end{question}

One potential route to answering these questions is investigation of previous proofs given in the literature.

\begin{question}
    Which parts of the argument given in the proof of \cite[Thm 1.1]{MR2677836} (specifically the construction of $\Psi=\Phi^{-1}$) break down when working over an arbitrary field?    
\end{question}

Taken together, Theorems \ref{thm:fiber sequence} and \ref{thm:what's the cokernel?} provide a generalization of \cite[Thm. 1.1]{MR2677836}.
In practice, if one wishes to construct a $G$-graded extension of a fusion category over $\mathbb K$, Corollary \ref{cor:5-term LES} is the most useful computation tool.
Unfortunately, it can often be difficult to determine the various maps in the exact sequence, as the following examples show.

\begin{example}\label{eg:Q-}
    There is a fusion category $\mathcal Q_-$ over $\mathbb R$ consisting of only two simple objects: $\1$ and $Y$.
    The unit is a real simple, while $\End(Y)$ is isomorphic to the quaternion algebra $\mathbb H$. 
    The fusion rules are determined by the fact that $Y\otimes Y\cong4\cdot\1$.
    This category was constructed explicitly as a non-split Tambara-Yamagami category over $\mathbb R$ in \cite{plavnik2023tambarayamagami}, and in that notation $\mathcal Q_-=\mathcal C_{\mathbb H}(\mathbf 1,\chi,-1/2)$, where $\chi$ is the unique nontrivial bicharacter on the trivial group.
    Upon base extension to $\mathbb C$, this category becomes $\Vec_{\mathbb C}(\mathbb Z/2\mathbb Z)$.
    This category was also discussed before in \cite[Ex. 3.11]{etingofDescentAndForms} and \cite[Rmk. 2.14]{johnson-freydSpinStatistics}.

    The Drinfeld center $\mathcal Z(\mathcal Q_-)$ has four simple objects: $\1,M,Y$, and $MY:=M\otimes Y$.
    As an object $M$ is just $\1$, but with a half-braiding that acts by $-1$ when braiding across $Y$.
    The object $Y$ is the same from $\mathcal Q_-$, but now equipped with a trivial half-braiding.
    The center is $(\mathbb Z/2\mathbb Z)^2$-graded, and so $\aut_{\otimes}(\id_{\mathcal Z(\mathcal Q_-)})\cong(\mathbb Z/2\mathbb Z)^2$.
    Since $\Br(\mathbb R)\cong\mathbb Z/2\mathbb Z$, the exact sequence of Corollary \ref{cor:5-term LES} becomes
    \[\mathbb Z/2\mathbb Z\hookrightarrow(\mathbb Z/2\mathbb Z)^2\to\mathbb Z/2\mathbb Z\to\brpic(\mathcal Q_-)\to\aut_{br}\big(\mathcal Z(\mathcal C)\big)\,.\]
    It follows that the map $\Br(\mathbb R)\to\brpic(\mathcal Q_-)$ must be zero.
    This means that the $\mathcal Q_-$ bimodule category $\mathcal Q_-\boxtimes(\mathbb H\text{-}\Mod)$ is equivalent to $\mathcal Q_-$ itself.
    This equivalence just swaps the roles of $\1$ and $Y$.
    
    Since $H^3(\mathbb R;\mathbb G_m)=1$, Corollary \ref{cor:H^3 condition} implies that $\brpic(\mathcal Q_-)\cong$ \linebreak$\aut_{br}(\mathcal Z(\mathcal Q_-))$.
    The group $\aut_{br}(\mathcal Z(\mathcal Q_-))$ is easier to calculate directly, though only marginally.
    If $(F,J)$ is a braided equivalence of $\mathcal Z(\mathcal Q_-)$, then it has to preserve endomorphism algebras and respect braidings.
    This implies that the underlying linear functor $F$ must fix all objects.
    The Skolem-Noether theorem implies that the action of $F$ on $\End(Y)\cong\mathbb H$ must be inner, and this is equivalent to saying that $F$ is naturally isomorphic to the identity.
    Thus, the only interesting data must lie in the tensorator $J_{A,B}:F(A)\otimes F(B)\to F(A\otimes B)$.

    Naturality of $J$ implies that it must commute with all quaternions, and therefore must be real-valued.
    Thus, the monoidal data can be computed as if this category were just $\Vec_{\mathbb R}((\mathbb Z/2\mathbb Z)^2)$.
    The potential tensorators are classified by $H^2((\mathbb Z/2\mathbb Z)^2;\mathbb R^\times)\cong(\mathbb Z/2\mathbb Z)^3$.
    The cohomology classes of the following cocycles
    \begin{align*}
        J^a_{M^iY^j,M^kY^\ell}&=(-1)^{ik}\cdot\id\,,\\
        J^b_{M^iY^j,M^kY^\ell}&=(-1)^{j\ell}\cdot\id\,,\;\text{and}\\
        J^c_{M^iY^j,M^kY^\ell}&=(-1)^{i\ell}\cdot\id\;,\;
    \end{align*}
    form a $\mathbb Z/2\mathbb Z$-basis for this space.
    Since the underlying functors are all the identity, composition of such functors corresponds to multiplying these scalar functions.
    Direct computation shows that $J^a$ and $J^b$ respect the braiding, while $J^c$ does not.
    Therefore, we can conclude that $\brpic(\mathcal Q_-)\cong\aut_{br}(\mathcal Z(\mathcal Q_-))\cong(\mathbb Z/2\mathbb Z)^2$.
\end{example}

\begin{example}\label{eg:Q+}
    Similar to the previous example, there is a category $\mathcal Q_+:=\mathcal C_{\mathbb H}(\mathbf 1,\chi,+1/2)$.
    This category has identical fusion rules to $\mathcal Q_-$, but the associator $(Y\otimes Y)\otimes Y\to Y\otimes (Y\otimes Y)$ differs by a sign.
    This category is a real form of $\Vec_{\mathbb C}^\omega(\mathbb Z/2\mathbb Z)$, where the associator is determined by the unique nonzero class in $\omega\in H^3(\mathbb Z/2\mathbb Z;\mathbb C^\times)\cong\mathbb Z/2\mathbb Z$.

    The Drinfeld center $\mathcal Z(\mathcal Q_+)$ has three simple objects: $\1,X$, and $V$.
    Both $\1$ and $X$ are real and invertible.
    The underlying object of $V$ is $Y$, and the half-braiding $V\otimes Y\to Y\otimes V$ is, up to isomorphism, multiplication by $i$ on the $V$ factor.
    Since $\End(V)$ consists of only those morphisms in $\End(Y)$ that commute with the half-braiding, it follows that $\End(V)\cong\mathbb C\subset\mathbb H\cong\End(Y)$.
    The fusion rules are determined by the formulas $X\otimes V\cong V$, and $V\otimes V\cong2\cdot(\1\oplus X)$.

    It is immediate from the fusion rules that $\mathcal Z(\mathcal Q_+)$ is $\mathbb Z/2\mathbb Z$-graded.
    Since there are exactly two invertible objects, and as before $H^3(\mathbb R;\mathbb G_m)=1$, the exact sequence from \ref{cor:5-term LES} becomes
    \[\mathbb Z/2\mathbb Z\hookrightarrow\mathbb Z/2\mathbb Z\to\mathbb Z/2\mathbb Z\to\brpic(\mathcal Q_+)\twoheadrightarrow\aut_{br}\big(\mathcal Z(\mathcal C)\big)\,.
    \]
    We can then conclude that the first map must be an isomorphism, so we have an exact sequence
    \[\mathbb Z/2\mathbb Z\hookrightarrow\brpic(\mathcal Q_+)\twoheadrightarrow\aut_{br}\big(\mathcal Z(\mathcal C)\big)\,.
    \]
    Thus we find that the $\mathcal Q_+$ bimodule given by $\mathcal Q_+\boxtimes(\mathbb H\text{-}\Mod)$ is not equivalent to the trivial bimodule.
    The two bimodules $\mathcal Q_+$ and $\mathcal Q_+\boxtimes(\mathbb H\text{-}\Mod)$ look incredibly similar: they are equivalent as left modules, and also as right modules!
    However, bimodule categories also have middle associators, and here these two bimodules differ by a sign.

    Close analysis of \cite[Thm. 6.10, \& Ex. 6.11]{plavnik2023tambarayamagami} reveals that $\mathcal Z(\mathcal Q_+)\simeq \mathcal C_{\mathbb C}(\mathbf 1,\id_{\mathbb C},\chi,-1/2)$.
    By \cite[Lemma 6.18]{GJS}, since there are no nontrivial automorphisms of $\Inv(\mathcal Z(\mathcal Q_+))\cong\mathbb Z/2\mathbb Z$, the monoidal automorphisms of $\mathcal Z(\mathcal Q_+)$ form a group isomorphic to $(\mathbb Z/2\mathbb Z)^2$.
    Here the functors are parameterized, up to monoidal equivalence by $F(\xi,\lambda)$, where $\xi\in\Gal(\mathbb C/\mathbb R)$ controls the action of the functor on $\End(V)$, and $\lambda\in\{1,i\}$.
    From our description of the half-braiding on $X$ and $V$, it follows that $\sigma(w)=1$ and $\sigma_3(1)=i$ in the notation of \cite[Proposition 6.19]{GJS} and therefore the only braided functors $F(\xi,\lambda)$ must satisfy $\lambda=1$ (by their Equation 65) and $\xi=\id_{\mathbb C}$ (by their Equation 66).
    In summary, $\aut_{br}(\mathcal Z(\mathcal Q_+))\cong1$, and therefore $\brpic(\mathcal Q_+)\cong\mathbb Z/2\mathbb Z$.
\end{example}

The two examples above indicate situations where it is hard to tell, but typically it should be expected that the map $\Br(\mathbb K)\to\brpic(\mathcal C)$ is nonzero.
To see why this is, consider a 2-separable fusion category $\mathcal C$ over $\mathbb K$.
In general, the collection of simple objects in $\mathcal C$ will have endomorphism algebras that are finite dimensional division algebras $\mathbb H_i$ over separable field extensions $\mathbb L_i$ of $\mathbb K$.
This collection forms a multiset of isomorphism classes of algebras that can be recorded as an element of the Brauer \emph{ring} $B(\mathbb K)$ (see for example \cite{MR850345}).
We will call this element the \emph{algebra profile} of $\mathcal C$, and denote it by $[\mathcal C]$.
For example, for the quaternion group $Q_8$, we have that $[\Rep_{\mathbb R}(Q_8)]=4[\mathbb R]+[\mathbb H]$.

If $[\mathbb D]\in\Br(\mathbb K)$ has trivial image in $\brpic(\mathcal C)$, this means that $\mathcal C\simeq\mathcal C\boxtimes(\mathbb D\text{-}\Mod)$ as $\mathcal C$ bimodule categories.
In particular, this implies that they must have the same algebra profiles, but this means that $[\mathcal C]\cdot[\mathbb D]=[\mathcal C]$ in $B(\mathbb K)$.
For this to happen, we would need $[\mathcal C]$ to be a sum of orbits under the action of the subgroup $\langle[\mathbb D]\rangle$.

In the previous examples, this is exactly what happened, because $[\mathcal Q_\pm]=[\mathbb R]+[\mathbb H]\in B(\mathbb R)$, but this is the exception, not the rule.
For example $[\Rep_{\mathbb R}(Q_8)]\cdot[\mathbb H]=4[\mathbb H]+[\mathbb R]\neq[\Rep_{\mathbb R}(Q_8)]$.

Usually, a quick glance at the algebra profile will be enough to distinguish the two categories.
However, when $[\mathcal C]\cdot[\mathbb D]=[\mathcal C]$, more data needs to be checked, as the example of $\mathcal Q_+$ demonstrates.

\begin{example}
    Let us momentarily work over the complex numbers $\mathbb K=\mathbb C$.
    The results of \cite{MR2677836} show that $G$-graded extensions of $\Vec_{\mathbb C}$ are classified by homotopy classes of maps $BG\to B\BrPic(\Vec_{\mathbb C})$.
    Since $H^*(\mathbb C;\mathbb G_m)$ is $\mathbb C^\times$ in degree $3$ and trivial otherwise, it follows that $B\BrPic(\Vec_{\mathbb C})$ is an Eilenberg-Mac Lane space of type $K(\mathbb C^\times,3)$.
    It follows that $G$-extensions of $\Vec_{\mathbb C}$ are in one-to-one correspondance with cohomology classes $H^3(G;\mathbb C^\times)$.
    This corresponds precisely\footnote{The classification of these pointed categories up to monoidal equivalence would involve quotienting out by the action of $\aut(G)$, but we are only interested here in equivalences that respect the $G$-grading.} with the classification of possible associators on $\Vec_{\mathbb C}(G)$.

    Working now over $\mathbb K=\mathbb R$, the classification of $G$-graded extensions of $\Vec_{\mathbb R}$ is in terms of homotopy classes of maps $BG\to B\BrPic(\Vec_{\mathbb R})$.
    The only nontrivial homotopy groups of $B\BrPic(\Vec_{\mathbb R})$ are $\pi_1=\Br(\mathbb R)=\mathbb Z/2\mathbb Z$, and $\pi_3=\mathbb R^\times$.
    The homotopy type of such a space is classified by a single Postnikov $k$-invariant $k^{4}\in H^4(\mathbb Z/2\mathbb Z;\mathbb R^\times)\cong\mathbb Z/2\mathbb Z$.
    By composing with the map $B\BrPic(\Vec_{\mathbb R})\to K(\mathbb Z/2\mathbb Z,1)$, we find that any map $BG\to B\BrPic(\Vec_{\mathbb R})$ determines a class $[f]\in H^1(G;\mathbb Z/2\mathbb Z)=\Hom(G,\mathbb Z/2\mathbb Z)$.
    In this way, $G$-graded extensions of $\Vec_{\mathbb R}$ correspond to homomorphisms $f:G\to\mathbb Z/2\mathbb Z$, together with a solution $\tilde{f}$ to the lifting problem below.
    \[\begin{tikzcd}[ampersand replacement=\&]
    	\& {B\mathcal Br\mathcal Pic(\mathrm{Vec}_{\mathbb R})} \\
    	BG \& {K(\mathbb Z/2\mathbb Z\,,\,1)} \& {K(\mathbb R^\times,\,4)}
    	\arrow[from=1-2, to=2-2]
    	\arrow["{\tilde{f}}", dotted, from=2-1, to=1-2]
    	\arrow["f"', from=2-1, to=2-2]
    	\arrow["{k^4}"', from=2-2, to=2-3]
    \end{tikzcd}\]
    Such lifts only exist if $[f^*(k^4)]\in H^4(G;\mathbb R^\times)$ is zero.
    Consider the universal example, where $G=\mathbb Z/2\mathbb Z$ and $f=\id_{BG}$.
    The lifting problem is asking if there exist $\mathbb Z/2\mathbb Z$-graded extensions of $\Vec_{\mathbb R}$ such that the nontrivial component is spanned by a quaternionic simple object.
    The categories $\mathcal Q_\pm$ from Examples \ref{eg:Q-} and \ref{eg:Q+} show that such extensions do in fact exist, and this means that the class $[\id^*(k^4)]=[k^4]\in H^4(\mathbb Z/2\mathbb Z;\mathbb R^\times)$ must be trivial.
    In other words, the classifying space $B\BrPic(\Vec_{\mathbb R})$ is a split extension of the form $K(\mathbb R^\times,3)\times K(\mathbb Z/2\mathbb Z,1)$.

    From this analysis, it follows that all $G$-graded extensions of $\Vec_{\mathbb R}$ are classified by pairs $(f,\varphi)$, where $f:G\to\mathbb Z/2\mathbb Z$, and $\varphi\in H^3(G;\mathbb R^\times)$.
    The fusion categories constructed as such extensions will have one simple object $X_g$ for every $g\in G$.  The algebra $\End(X_g)$ will be $\mathbb R$ if $f(g)=0$, and $\mathbb H$ if $f(g)=1$.
    The fusion rules will be
    \[X_g\otimes X_h\cong 4^{f(g)\cdot f(h)}\cdot X_{gh}\,,\]
    and the associator will be determined by the 3-cocycle $\varphi$.
    Upon base extension to $\mathbb C$, categories of this form become pointed fusion categories, and thus they are Galois-twisted real forms of pointed categories in the sense of \cite{etingofDescentAndForms}.
\end{example}

\begin{remark}
    It would be interesting to determine the homotopy type of $\BrPic(\Vec_{\mathbb K})$ for each field $\mathbb K$.
    The above example shows that this would be equivalent to computing the Postnikov classes $k^4\in H^4(\Br(\mathbb K);\mathbb K^\times)$.
    The only reason we were able to verify triviality of $k^4$ for $\mathbb R$ was that the categories $\mathcal Q_\pm$ had already been shown to exist in \cite{plavnik2023tambarayamagami}.
    
\end{remark}

\begin{example}\label{eg:BrPic of Q-}
    From Example \ref{eg:Q-}, we know that the invertible bimodules for $\mathcal Q_-$ form a Klein-four group that is in bijective correspondence with $\aut_{br}(\mathcal Z)$.
    What do these bimodules look like?

    Since $\mathcal Q_-$ itself is $\mathbb Z/2\mathbb Z$-graded, there is a nontrivial monoidal autoequivalence $(F,J):\mathcal Q_-\to\mathcal Q_-$ that acts by the identity functor, and has $J_{Y,Y}=-\id_{Y\otimes Y}$.
    Note that this is only nontrivial because we are working over $\mathbb R$, since $H^2(\mathbb Z/2\mathbb Z;\mathbb R^\times)\cong\mathbb Z/2\mathbb Z$, whereas $H^2(\mathbb Z/2\mathbb Z;\mathbb C^\times)=0$.
    Using this autoequivalence, we can form an invertible bimodule $(\mathcal Q_-)_F$, with the right action twisted by $F$.
    Such bimodules that are equivalent as modules to the regular left module are called outer, and they are generally easier to understand than generic invertible bimodules.
    Inducing this autoequivalence $(F,J)$ up to the center, we find that this corresponds to the equivalence $(\Id_{\mathcal Z},J^b)$ from Example \ref{eg:Q-}.
    
    What about the other functor $A:=(\Id_{\mathcal Z},J^a)$?
    We can find the corresponding invertible bimodule by using the folding trick as in the proof of Theorem \ref{thm:what's the cokernel?}.
    We know that $\mathcal T_A$ is Morita trivial, and any invertible bimodule from $\mathcal T_A$ to $\Vec_{\mathbb R}$ can be `unbent' to provide the desired bimodule category for $\mathcal Q_-$.
    
    Consider a simple tensor $U\boxtimes_{\mathcal Z_{A}}V$ in $\mathcal T_A=(\mathcal Q_-)^{mp}\boxtimes_{\mathcal Z_A}\mathcal Q_-$.
    Since every object of $\mathcal Q_-$ admits a lift to $\mathcal Z$, we find that $U\boxtimes_{\mathcal Z_A}V\cong\1\boxtimes_{\mathcal Z_A}UV$, and so all simple objects in $\mathcal T_A$ can be written as summands of either $\1\boxtimes_{\mathcal Z_A}\1$ or $\1\boxtimes_{\mathcal Z_A}Y$.
    The real algebra $\End_{\mathcal T_A}(\1\boxtimes_{\mathcal Z_A}\1)$ is generated under composition by a single morphism
    \[f:=\big(\1\boxtimes_{\mathcal Z_A}\1\to\1\boxtimes_{\mathcal Z_A}(M\otimes\1)\to(\1\otimes M)\boxtimes_{\mathcal Z_A}\1\to\1\boxtimes_{\mathcal Z_A}\1\big)\,,\]
    which is only determined up to a real scalar multiple.
    The composition $f^2$ can be reduced using string diagram calculus to be a scalar multiple of the identity.
    This reduction involves the tensor product of two copies of $M$, and thus produces a factor of $J^a_{M,M}=-1$, and so $f^2=-\lambda\cdot\id$.
    Adjusting $f$ by a real scalar factor only changes $f^2$ by a positive scalar factor, and thus we may assume that $f^2=-\id$.
    Thus we can conclude that $\End(\1\boxtimes_{\mathcal Z_A}\1)\cong\mathbb C$.

    A similar argument shows that $\End(\1\boxtimes_{\mathcal Z_A}Y)$ is isomorphic to $\mathbb C\otimes\mathbb H\cong M_2(\mathbb C)$.
    This means that $\mathcal T_A$ has two simple objects that are both complex.
    The only Morita trivial multifusion category over $\mathbb R$ of this form is the category of complex bimodules $\Bim_{\Vec_{\mathbb R}}(\mathbb C)$ from Example \ref{eg:Bim(C)}.
    The invertible bimodule from $\mathcal T_A$ to $\Vec_{\mathbb R}$ is thus equivalent to $\Vec_{\mathbb C}$.
\end{example}

\begin{example}
    The homotopy groups of $\BrPic(\mathcal Q_-)$ are all 2-torsion, so if we try to build an extension $BG\to B\BrPic(\mathcal Q_-)$ with $|G|$ odd, then the resulting category will just be $\mathcal Q_-\boxtimes\Vec_{\mathbb R}(G)$.
    In light of this, let us try to build an interesting $G=\mathbb Z/2\mathbb Z$ extension.
    Let $\mathcal X$ denote the invertible $\mathcal Q_-$ bimodule category $\Vec_{\mathbb C}$ from the previous example, and let $X$ be the complex simple object that generates $\mathcal X$.
    Assigning $1\mapsto\mathcal X$ determines a homomorphism $G\to\brpic(\mathcal Q_-)$, and this determines a map $f^{(1)}$ from the 1-skeleton $BG^{(1)}$ to $B\BrPic(\mathcal Q_-)$ that we can use to begin constructing our extension.

    Since we used a homomorphism, $f^{(1)}$ can be extended to the 2-skeleton of $BG$.
    Since $BG\simeq K(\mathbb Z/2,1)\simeq\mathbb RP^\infty$ has a model with exactly one cell in each dimension, at each stage of this process, we only need to choose where a single cell goes, then determine the obstruction.
    The unique 2-cell in which we are interested is the relation $1+1=0$ in $G$, which corresponds to choosing a bimodule equivalence $M_{1,1}:\mathcal X\boxtimes_{\mathcal Q_-}\mathcal X\to\mathcal Q_-$.
    In $\Aut_{br}(\mathcal Z)$, this corresponds to choosing a monoidal natural isomorphism $A\circ A\to\Id_{\mathcal Z}$.
    Here we can choose the identity, because $A^2=\Id_{\mathcal Z}$ on the nose.
    
    The inclusion $\pi_1\Phi:\mathbb Z/2\mathbb Z\to(\mathbb Z/2\mathbb Z)^2$ is split with trivial action, so the induced map on coefficients $(\pi_1\Phi)_*:H^3(G;\Inv(\mathcal Z))\to$\linebreak$ H^3(G;\aut_{\otimes}(\Id_{\mathcal Z}))$ must also be an injective.
    We must verify that the obstruction $O_3\in H^3(G;\Inv(\mathcal Z))$ is trivial, but all values of $(\pi_1\Phi)_*O_3$ are just a compositions of identity natural transformations, so $O_3$ must be trivial by injectivity of $(\pi_1\Phi)_*$.
    It follows that our map can be extended to the three skeleton of $BG$.

    The $O_4$ obstruction is also trivial, though at this time, the author is unaware of any technique more convenient that direct computation.
    Easier to compute is the Pontryagin-Whitehead map $PW:H^2(G;\Inv(\mathcal Z))\to H^4(G;\mathbb R^\times)$.
    Since $G$ is finite, $\mathbb R^\times$ coefficients are the same as $\mathbb Z/2\mathbb Z$ coefficients.
    The action of $G$ on $\mathcal Z^\times\simeq\Vec_{\mathbb R}(\mathbb Z/2\mathbb Z)$ is by the (restriction of) the functor $A$.
    The nontrivial class in $H^2$ is represented by the cocycle $\varphi(g,h)=M^{gh}$.
    The definition of $PW(\varphi)$ given in \cite[Def. 8.11]{MR2677836} involves a composition of maps, all but two of which can be made trivial.
    The two nontrivial maps involve the action of the tensorator $J^a$ on $\varphi$, and they result in factors $(-1)^{(g+h)fghk}$ and $(-1)^{(h+k)fghk}$.
    The product becomes
    \[(-1)^{fghk(g+2h+k)}=(-1)^{fghk(g+k)}=(-1)^{fg^2hk+fghk^2}=(-1)^{2fghk}\;=\;1\,.\]
    From this it follows (see \cite[Prop. 8.15]{MR2677836}) that the $O_4$ obstruction will remain trivial upon changing the equivalence $M_{1,1}$.
    
    Thus, starting from the homomorphism $G\to\brpic(\mathcal Q_-)$, there are four options: two choices for the tensor products $M_{1,1}$ (a torsor over $H^2(G;\mathbb Z/2\mathbb Z)$), and two choices for the associators (a torsor over $H^3(G;\mathbb R^\times)$).
    All of these result in distinct $\mathbb Z/2\mathbb Z$-graded extensions of $\mathcal Q_-$.
    One of these extensions was described in \cite[Apx. A]{MR4934638}, where it arose as a braided real form of $\mathcal Z(\Vec_{\mathbb R}(\mathbb Z/2\mathbb Z)$.
    This category was shown to be nondegenerately braided, but not a center, and therefore it represents a (conjecturally unique!) nontrivial class in $\ker(\mathcal Witt(\Vec_{\mathbb R})\to\mathcal Witt(\Vec_{\mathbb C}))$.
\end{example}

\section{A conjectural generalization of the fiber sequence}\label{sec:Conjectural extension}

In this section, which is independent of our main results, we describe a conjecture due to Jones and Reutter that, if true, would generalize the exact sequence of Theorem \ref{thm:what's the cokernel?}.

In \cite[Thm. 3.4]{MR3354332}, Grossman, Jordan, and Snyder establish a homotopy fiber sequence
\[\mathcal C^\times\to\Aut_{\otimes}(\mathcal C)\to\Out(\mathcal C)\,,\]
where $\Out(\mathcal C)$ is the subgroupoid of $\BrPic(\mathcal C)$ consisting of all bimodules that are trivial as left module categories.
This result is a categorification of the Rosenberg-Zelinsky exact sequence (see \emph{e.g.} \cite{MR2658180} and \cite{galindoTensorFunctorsMorita2016}) for algebras, which itself is a generalization of the classical Inn-Aut-Out exact sequence:
\[Z(G)\hookrightarrow G\to\aut(G)\twoheadrightarrow\out(G)\]
for groups $G$.
Given this persistent pattern, it is reasonable to suspect this result can be categorified further.
The following conjecture was suggested to me by Corey Jones.

\begin{conjecture}[\cite{jonesPrivCom}]\label{conj:RZForF2Cs}
    Under suitable assumptions on a monoidal $\mathbb K$-linear 2-category $\mathfrak C$, there is a homotopy fiber sequence
    \[\mathfrak C^\times\to\Aut_{\otimes}(\mathfrak C)\to\Out(\mathfrak C)\,.\]
\end{conjecture}

If this conjecture is true, it would generalize Theorems \ref{thm:fiber sequence} and \ref{thm:what's the cokernel?} in the following sense.

\begin{corollary}[subject to Conjecture \ref{conj:RZForF2Cs}]
    If $\mathfrak C=\Mod(\mathcal B)$ for $\mathcal B$ a braided (separable) fusion 1-category over $\mathbb K$, then there is a fiber sequence
    \[\Pic(\mathcal B)\to\Aut_{br}(\mathcal B)\to\Out(\mathfrak C)\,.\]
    In particular, if $\mathcal B=\mathcal Z(\mathcal C)$ for some fusion 1-category $\mathcal C$ with $\Omega\mathcal Z\mathcal C=\mathbb K$, this becomes
    \[\BrPic(\mathcal C)\to\Aut_{br}\big(\mathcal Z(\mathcal C)\big)\to\Out(\mathfrak C)\,,\]
    which is a delooping of the fiber sequence in Theorem \ref{thm:fiber sequence}.
    Furthermore, the long exact sequence on homotopy groups recovers the extension of the fiber sequence described in Theorem \ref{thm:what's the cokernel?}.
\end{corollary}

\begin{proof}[proof (sketch)]
    The invertible objects in $\mathfrak C=\Mod(\mathcal B)$ are precisely the invertible module categories for $\mathcal B$, \emph{aka} $\Pic(\mathcal B)$.
    If $\mathcal B=\mathcal Z(\mathcal C)$, then $\Pic(\mathcal Z(\mathcal C))\simeq\BrPic(\mathcal C)$ by the invertibility of $\mathcal C$ as a 1-morphism in $\M$, which follows from Theorem \ref{thm:main invertibility theorem}.

    Autoequivalences of $\Mod(\mathcal B)$ are invertible $\mathcal B$ bimodule categories by (a higher version of) 
    Eilenberg-Watts.
    For the autoequivalence to be monoidal, it must be of the form $\mathcal B_F$, where $F$ is a braided autoequivalence of $\mathcal B$.
    
    By definition, $\pi_1\Out(\Mod(\mathcal Z(\mathcal C)))$ consists of the autoequivalences of $\Mod(\mathcal Z(\mathcal C))$ as a bimodule over itself.
    This is the same as the group of invertible objects (up to equivalence) in the center of $\Mod(\mathcal Z(\mathcal C))$.

    Since $\mathcal Z(\mathcal C)$ is Witt-trivial, it follows that $\Mod(\mathcal Z(\mathcal C))$ is Morita equivalent to $2\Vec_{\mathbb K}$.
    This implies that $\mathcal Z(\Mod(\mathcal Z(\mathcal C)))$ is braided 2-equivalent to $\mathcal Z(2\Vec_{\mathbb K})\simeq2\Vec_{\mathbb K}$,
    so we find that
    \[\pi_1\Out(\Mod(\mathcal Z(\mathcal C)))\simeq\Inv(2\Vec_{\mathbb K})\simeq\Br(\mathbb K)\,.\]

    By construction of the original fiber sequence, the map $\Aut_{\otimes}(\mathfrak C)\to\Out(\mathfrak C)$ is $\mathcal F\mapsto\mathfrak C_{\mathcal F}$, with the right bimodule structure twisted by $\mathcal F$.
    We already know that $\mathfrak C=\Mod(\mathcal B)$ implies that $\mathcal F\simeq\mathcal B_F\otimes_{\mathcal B}(-)$ for some braided equivalence $F$ of $\mathcal B$.
    When $\mathcal B=\mathcal Z=\mathcal Z(\mathcal C)$, we are just tensoring with $\mathcal Z_F$ on the left.
    Conjugating by the Witt equivalence $\mathcal C:\mathcal Z\to\Vec_{\mathbb K}$ shows that this map is equivalent to the operation $\mathcal T_F\bt(-)$ of tensoring with the invertible category $\mathcal T_F$ of Theorem \ref{thm:what's the cokernel?}.
\end{proof}

\bibliography{main}{}
\bibliographystyle{alpha}

\end{document}

%% file: TestStyle.tikzstyles
\tikzstyle{wc}=[fill=white, draw=black, shape=circle]
\tikzstyle{ws}=[fill=white, draw=black, shape=rectangle]
\tikzstyle{bd}=[fill=black, draw=black, shape=circle, inner sep=0pt, minimum size=4pt]
\tikzstyle{wd}=[fill=white, draw=black, shape=circle, inner sep=0pt, minimum size=4pt]
\tikzstyle{bt-up}=[fill=black, draw=black, shape=isosceles triangle, inner sep=0pt, minimum size=4pt, shape border rotate=90]
\tikzstyle{bt-down}=[fill=black, draw=black, shape=isosceles triangle, inner sep=0pt, minimum size=4pt, shape border rotate=-90]
\tikzstyle{wt-up}=[fill=white, draw=black, shape=isosceles triangle, shape border rotate=90, inner sep=0pt, minimum size=4pt]
\tikzstyle{wt-down}=[fill=white, draw=black, shape=isosceles triangle, shape border rotate=-90]
\tikzstyle{box}=[fill=white, draw=black, shape=rectangle, minimum width=0.5cm, minimum height=0.3cm]
\tikzstyle{widebox}=[fill=white, draw=black, shape=rectangle, minimum width=1.0cm, minimum height=0.3cm]
\tikzstyle{vwidebox}=[fill=white, draw=black, shape=rectangle, minimum width=1.5cm, minimum height=0.3cm]
\tikzstyle{wdiam}=[fill=white, draw=black, shape=diamond]
\tikzstyle{oc}=[fill={rgb,255: red,255; green,128; blue,0}, draw=none, shape=circle, inner sep=0pt, minimum size=4pt]

\tikzstyle{red}=[-, draw=red]
\tikzstyle{blue}=[-, draw=blue]
\tikzstyle{green}=[-, draw=green]
\tikzstyle{green}=[-, draw=green]
\tikzstyle{dash}=[-, dashed]
\tikzstyle{dot}=[-, dotted]